\documentclass[times,sort&compress,3p]{elsarticle}
\journal{Journal of Multivariate Analysis}
\usepackage[labelfont=bf]{caption}

\usepackage{amsmath,amsfonts,amssymb,amsthm,bbm,booktabs,color,epsfig,graphicx,hyperref}

\theoremstyle{plain}
\newtheorem{theorem}{Theorem}

\newtheorem{proposition}{Proposition}
\newtheorem{lemma}{Lemma}
\newtheorem{corollary}{Corollary}
\newtheorem{definition}{Definition}

\input{preamble.tex}

\begin{document}

\begin{frontmatter}

\title{Consistency of empirical distributions of sequences of graph statistics in networks
with dependent edges} 

\author[1]{Jonathan R. Stewart\corref{mycorrespondingauthor}}

\address[1]{Department of Statistics \\ Florida State University \\ 
117 N Woodward Ave\\
Tallahassee, FL 32306-4330}

\cortext[mycorrespondingauthor]{Corresponding author. Email address: \url{jrstewart@fsu.edu}.}

\begin{abstract}
One of the first steps in applications of statistical network analysis is frequently
to produce summary charts of important features of the network.
Many of these features take the form of sequences of graph statistics counting the number of realized events in the network,
examples of which include the degree distribution,
as well as the edgewise shared partner distribution,
and more.
We provide conditions under which the empirical distributions of sequences of graph statistics
are consistent in the $\ell_{\infty}$-norm in settings where edges in the network are dependent.
We accomplish this task by deriving concentration inequalities that bound probabilities of deviations of graph statistics
from the expected value under weak dependence conditions.
We apply our concentration inequalities to empirical distributions of sequences of graph statistics
and derive non-asymptotic bounds on the $\ell_{\infty}$-error which hold with high probability.
Our non-asymptotic results are then extended to demonstrate uniform convergence almost surely
in selected examples.
We illustrate theoretical results through examples,
simulation studies,
and an application.
\end{abstract}

\begin{keyword} 
Empirical distributions of graph statistics \sep 
network data \sep 
statistical network analysis
\MSC[2020] Primary 62H12 \sep
Secondary 62G30
\end{keyword}

\end{frontmatter}

\section{Introduction\label{sec:1}}

We consider simple random graphs $\bX$ which are defined on a set of $N \geq 3$ nodes,
which we take without loss of generality to be the set $\mN \coloneqq \{1, \ldots, N\}$ throughout.  
The edge variables in $\bX$ are then given by 
\beno
X_{i,j} 
\= \begin{cases}
1 & \mbox{nodes $i$ and $j$ are connected in the graph} \\ 
0 & \mbox{otherwise}
\end{cases},
&& (i,j) \in \mN \times \mN. 
\ee
We assume that $X_{i,i} = 0$ for all $i \in \mN$ with probability $1$,
and in the case of undirected graphs, 
we assume that $X_{i,j} = X_{j,i}$  for all $\{i,j\} \subset \mN$ with probability $1$.  
We denote the support of $\bX$ by $\mbX$ 
and throughout consider probability spaces 
$(\mbX, \mP(\mbX), \mbP)$,
where $\mP(\mbX)$ is the power set of $\mbX$
and denote the corresponding expectation operator by $\mbE$. 

In this work, we will be interested in the empirical distributions of sequences of graph statistics 
defined around sequences of events. 
We consider sequences of mutually exclusive events 
$\mG_{0,m}, \mG_{1,m}, \ldots, \mG_{p,m}$ ($m \in \{1, \ldots, M\}$)
defined around the random graph $\bX$ 
and a corresponding sequence of graph statistics $\bs : \mbX \mapsto \mbR^{p+1}$
which are defined to be 
\be
\label{eq:stat}
s_{k}(\bX)
&\coloneqq& \dsum_{m=1}^{M} \, \mathbbm{1}(\mG_{k,m}),
&& k \in \{0,1, \ldots, p\}. 
\ee
The corresponding empirical distribution $\empnox : \mbX \mapsto [0,1]^{p+1}$ is then defined to be 
\be
\label{eq:emp}
\empk
&\coloneqq& \dfrac{1}{M} \, s_k(\bX),
&& k \in \{0, 1, \ldots, p\}. 
\ee
A 
prime 
example is the degree distribution.
Let $\mG_{d,i}$ be the event that node $i \in \mN$ has degree $d \in \{0, 1, \ldots, N-1\}$.
Then 
\be
\label{eq:deg_ex}
s_{d}(\bX)
\= \dsum_{i=1}^{N} \, \one(\mG_{d,i})
\=  \dsum_{i=1}^{N} \, \one\left(\, \dsum_{j \in \mN \setminus \{i\}} \, X_{i,j} \,=\,d \right),
&& d \in \{0, 1, \ldots, N-1\}, 
\ee 
in which case 
\beno
\empd 
\= \dfrac{s_d(\bX)}{N} &\in& [0, 1],
&& d \in \{0, 1, \ldots, N-1\}.
\ee
In words,
$s_{d}(\bX)$ counts the number of nodes with exactly $d \in \{0, 1, \ldots, N-1\}$ connections
to other nodes in the network $\bX$
and $\empd$ represents the proportion of nodes with degree precisely equal to $d$ in the network $\bX$.
We visualize an example of a network and corresponding empirical degree distribution in Fig. \ref{fig:ex}.
In this example,
observe that $\dim(\emp) = N$.
This work considers scenarios in which the dimension of the vectors $\emp$
encoding empirical distributions of sequences of graph statistics are allowed to grow unbounded with the size of the graph $N$.

\begin{figure}
\centering 
\includegraphics[scale = .4]{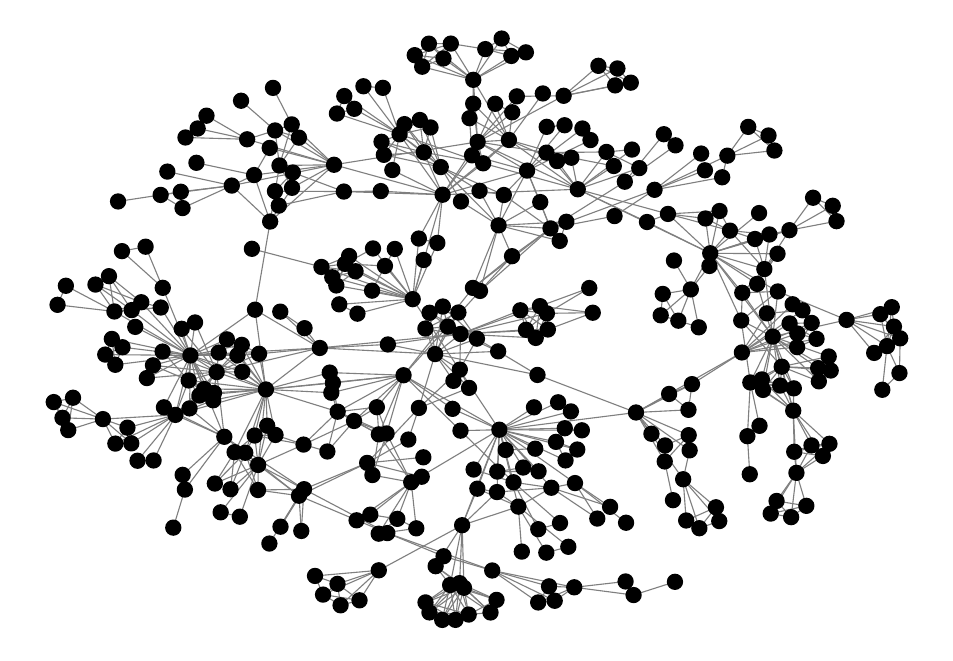} \hspace{.2cm}  
\includegraphics[scale = .4]{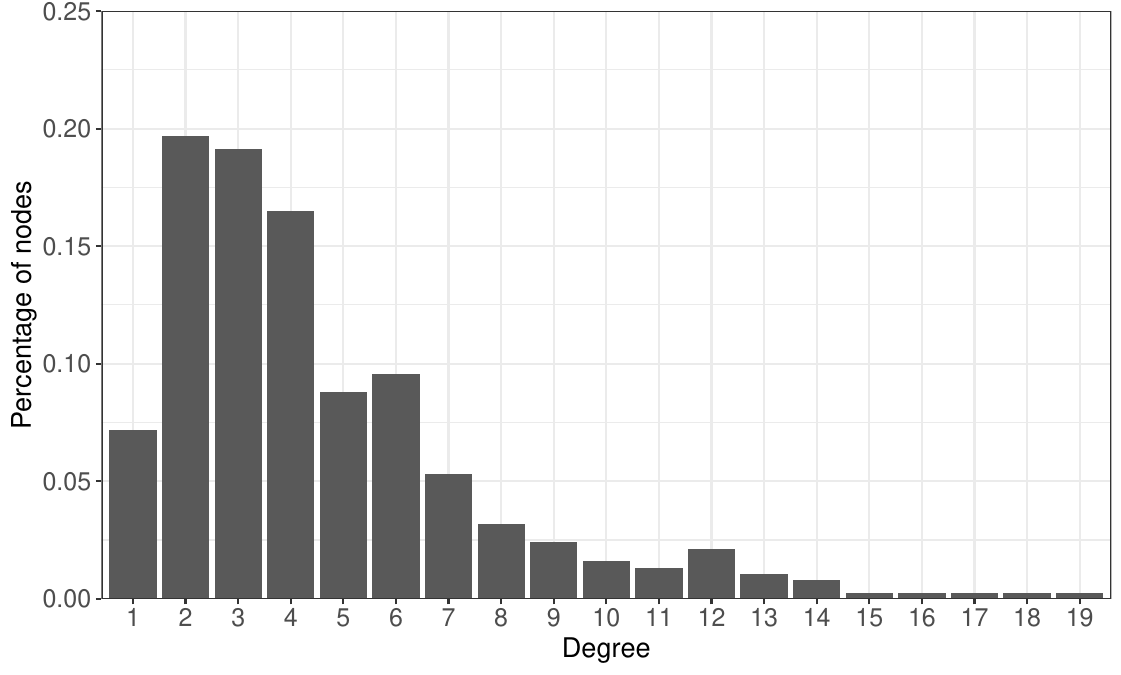}
\caption{\label{fig:ex} (left) A visualization of a collaboration network which consists of a set of researches as nodes,
with edges corresponding to co-authorship. (right) The empirical degree distribution of the collaboration network. 
This  network data set is maintained by \citet{nr}.}
\end{figure}

It is natural to ask under what conditions can we expect $\emp$ to provide an 
accurate estimate of the true underlying distribution of the sequence of graph statistics. 
We define this distribution to be $\truth : \mbX \mapsto [0, 1]^{p+1}$,
where 
\beno 
\truthk
&\coloneqq& \mbE \, \empk
\= \dfrac{1}{M} \, \dsum_{m=1}^{M} \, \mbP(\mG_{k,m}), 
&& k \in \{0, 1, \ldots, p\},
\ee
represents the theoretical marginal probabilities.  
If the indicator random variables $\one(\mG_{k,m})$ ($m \in \{1, \ldots, M\}$) 
are exchangeable, 
then $\truthk = \mbP(\mG_{k,m})$ for all $m \in \{1, \ldots ,M\}$,
which is analogous to the setting of an empirical distribution based on a random sample.
A notable difference in this work is that we will be considering 
settings where we obtain only a single observation of a network. 
As such, 
we do not have the benefit of replication 
and empirical distributions of graph statistics will be based on only a single observation of the network.
The interpretation of what would be the 
true distribution $\truth$ is then slightly different in this context. 
When considering the degree distribution, 
we can understand the marginal probability $\truthk \in [0, 1]$ to represent the probability 
that a randomly selected node $i \in \mN$ in the network will have degree equal to $k$. 
In a broader context, 
the results in this work establish conditions under which the empirical distributions $\emp$ 
of sequences of graph statistics will be stable, 
in the sense that deviations $\norm{\emp - \truth}_{\infty}$ 
will be small with high probability provided the size of the network $N$ is sufficiently large.

One of the main challenges of this problem lies in the fact that the 
random variables $\empo, \ldots, \empp$ will generally be dependent, 
even when the edge variables in the random graph are independent,
and the derivation of concentration inequalities for dependent random variables is highly non-trivial.  
A case in point is the degree distribution, 
as even if the edge variables are independent,
the degrees of nodes $i \in \mN$ and $j \in \mN \setminus \{i\}$ 
will still be dependent as both depend on the value of edge variable $X_{i,j}$. 
While cases of independent edge variables might be handled,
for example, 
with the bounded difference inequality,
the assumption that edge variables in a network are independent can be too strong for many applications, 
especially for applications in social network analysis 
where it has long been observed that edges are dependent \citep[e.g.,][]{Holland1972,Frank1980}. 
To overcome this challenge, 
we develop concentration inequalities for random graphs with dependent edges 
under different dependence conditions, 
which enables us to cover a wide range of applications. 
We demonstrate the applicability of our theory through mathematical examples
which are presented as corollaries and through simulation studies and a network data application. 

\s

The main contributions of this work include:
\ben
\item Deriving non-asymptotic bounds on the error $\norm{\emp - \truth}_{\infty}$
which hold with high probability; \s
\item Establishing a form of uniform convergence by demonstrating that $\norm{\emp - \truth}_{\infty}$
converges almost surely to $0$ as $N \to \infty$
in various theoretical applications which demonstrate the applicability of the main results; \s
\item Conducting simulation studies that showcase the empirical performance of the theoretical results; \ls
\item Demonstrating the theoretical results through an application to a school classes network data set,
which facilitates an exploration of rates of convergence through a specific sampling mechanism.
\een

The rest of the paper is organized as follows.
Section \ref{sec:related} reviews related work.
Theoretical results are presented in Section \ref{sec:2},
with simulation studies and empirical results being presented in Section \ref{sec:3}.
We present an application of our theory to a network data set in Section \ref{sec:application},
and conclude with a short discussion of the contributions in Section \ref{sec:5}.

\subsection{Related work\label{sec:related}}

\hide{
This work establishes the first results which prove the consistency of empirical distributions of sequences of graph statistics
for a general class of random graphs which allow edges to be dependent. 
Notably, 
this work covers a wide range of sequences of graph statistics,
whereas much of the existing literature focuses on specific sequences, 
predominantly the degree distribution of a network. 
We review related work which is closely related to the problem studied here following two main approaches. 
}

We review related work which is closely related to the problem studied here following two main approaches.

There are a number of works 
which establish limiting distributions of sequences of graph statistics,
facilitating inference on the distributions of sequences of graph statistics through an asymptotic approximation. 
Along this vein,
some examples include work by  
\citet{Krivitsky2011}, 
who established the limiting degree distribution of a class of sparse Bernoulli random graph models,
and \citet{Britton2020}, 
who established (among other theoretical results) the limiting degree distribution of 
directed preferential attachment models,
building on other work within this class of models \citep[e.g.,][]{Bollobas2003}.  
We take a different approach in this work and 
focus on establishing the consistency of empirical distributions of sequences of graph statistics as a means of 
facilitating inference on the distributions of sequences of graph statistics,
in contrast to using asymptotic approximations of distributions.  

Along a different inferential goal, 
there are a number of works which aim to estimate unknown degree distributions 
(or other quantities) 
of large networks through sampling. 
Examples include works by 
\citet{Antunes2021}, 
\citet{Zhang2015}, and 
\citet{Ribeiro2012}.
The inferential goal of these works is distinct from the goal of this work 
and that of the works cited in the previous paragraph. 
This distinction may be characterized as the differences between finite population versus super population inference 
in statistical network analysis applications
\citep{Schweinberger2020},
which can be understood in the following way. 
The work of 
\citet{Antunes2021} aims to infer an unknown,
but fixed,
degree distribution of a large network via sampling
within a finite population inference framework under which the entire network is the population of interest.  
In contrast,
the work of \citet{Britton2020} characterizes the limiting degree distribution of a certain class of networks,
where within a super population inference framework,
the population of interest is the population of degree distributions
which describes the variability of node degrees under different realizations of the network from a data-generating probability distribution.
In this work,
we will operate under a super population inferential framework,
aiming to characterize the statistical variability
of empirical distributions of sequences of graph statistics that would arise if we were able to replicate
the network from some data-generating probability distribution.

Other related works include that of \citet{Bickel2011} and \citet{Chan2014}, 
both of which considered the problem of fitting a class of
statistical models which assume edge variables are conditionally independent 
using empirical quantities related to graph statistics. 
The work of \citet{Bickel2011} 
introduced a method of fitting a class of
statistical models which assume edge variables are conditionally independent using a method of moments estimator 
based on empirical frequencies of graph statistics,
and as part of this work, 
established the asymptotic consistency of empirical quantities for degree distributions within this class of models. 
The work of \citet{Chan2014} proposed a consistent histogram estimator for graphons
which is based on a sorting algorithm of the empirical degree distribution. 
Both works are concerned with empirical quantities related to graph statistics,
namely the degree distribution, 
but in the context of an overall goal of developing methods for fitting statistical models to observed networks.

Related works on these topics have predominantly focused on studying the degree distributions of networks.
In this work,
we develop new theoretical results which cover a broad range of sequences of graph statistics,
including degree distributions and edgewise shared partner distributions as examples.
A key difference between this work and the cited related work is that the theory developed in this work
covers a broader class of random graphs by allowing edges within networks to be dependent.
The above-cited works make strong assumptions on the dependence structure of the edge variables in the network,
either assuming that edges are independent or conditionally independent.
In addition,
the cited results focus on asymptotic theory,
whereas our main results are non-asymptotic and establish uniform rates of convergence,
covering both a broad scope of distributions of sequences of graph statistics in settings where 
the edge variables are dependent.

\section{Theoretical results\label{sec:2}}

The main theoretical results are presented in Section \ref{sec:main_res}.
We study two applications of our theory, 
which are the 
degree distribution and the edgewise shared partner distribution, 
in Sections \ref{sec:deg_app} and \ref{sec:esp_app},
respectively. 
Before presenting these results,  
we first outline the key assumptions of this work,
which are weak dependence assumptions, 
and derive the  
concentration inequalities used  for the proofs of the main results 
in Section \ref{sec:concentration}. 
We discuss our weak dependence assumptions in further detail in 
Section \ref{sec:weak_dep},
emphasizing the applicability to real world networks.

\subsection{Concentration inequalities for random graphs with dependent edges\label{sec:concentration}} 

\hide{
We aim to study probabilities of the event 
$\norm{\emp - \truth}_{\infty} \geq \epsilon$ for $\epsilon > 0$ 
in order to establish rates of convergence for the empirical distributions of sequences of graph statistics. 
As discussed in Section \ref{sec:1},
a key challenge in network data applications lies in the fact that the networks of our world often possess dependent edges.}

We present two approaches to deriving concentration inequalities for quantities  $\norm{\emp - \truth}_{\infty}$
for random graphs with dependent edges.
One approach is  based on martingale decompositions,
whereas the other is based on covariances. 
Related approaches to developing concentration inequalities for functions of dependent random variables 
with countable supports 
based on works by \citet{Chazottes2007} and \citet{Kontorovich2008}
have been  
successfully applied in the statistical network analysis literature 
in settings of random graphs with dependent edges  
\citep{SchweinbergerStewart2020,Stewart2020}. 
For this work,
however,  
such approaches will not yield suitable bounds. 
We derive concentration inequalities 
for the explicit purpose of establishing
bounds on the tail probabilities of events $\norm{\emp - \truth}_{\infty} \geq \epsilon$ for $\epsilon > 0$
in Lemma \ref{lem:concentration}.

We define the collection of  (dependent) 
Bernoulli random variables for events $\mG_{k,i}$ 
by 
\be
\label{eq:berns}
B_{k,i} 
\;\coloneqq\;  \one(\mG_{k,i}) 
\;\in\; \{0, 1\},
&& 
(k,i) \in \{0, 1, \ldots, p\} \times \{1, \ldots, M\}, 
\ee
and 
let $d_{\tv}$ denote the total variation distance 
between two probability measures defined on a common measurable space. 
Central to the theoretical results developed in this work are 
three key conditions  
which are weak dependence conditions that quantify,
through different methods,  
the influence of any $B_{k,i}$ on any other $B_{k,j}$ ($j \in \{1, \ldots, M\} \setminus \{i\}$). 

\begin{definition}
\label{def1}
Let $\mbC$ denote the covariance operator corresponding to the probability distribution $\mbP$ 
and define
\hide{ 
\beno
\mcC_{i,k}
&\coloneqq& \dsum_{j \in \{1, \ldots, M\} \setminus \{i\}} \, \cov(B_{k,i}, \, B_{k,j}),
&& i \in \{1, \ldots, M\}, \; k \in \{0, 1, \ldots, p\}, 
\ee
and 
}
\beno
\mcC_N 
&\coloneqq& \dfrac{1}{M} \; 
\dsum_{i=1}^{M} \, 
\dsum_{j \in \{1, \ldots, M\} \setminus \{i\}} \, 
\dsum_{k=0}^{p} \, 
\cov(B_{k,i}, \, B_{k,j}). 
\ee
\end{definition}

Definition \ref{def1} 
quantifies the strength of dependence among the events of interest through a measure of average covariance,
allowing certain $B_{k,i}$ and $B_{k,j}$ to be highly correlated,  
provided the average covariance as measured by $\mcC_N$ is not too large. 
This allows for variable influence among the random variables $B_{k,i}$ ($i \in \{1, \ldots, M\}$, $k \in \{0, 1, \ldots, p\}$)
and some potentially strong dependence among the variables,
reiterating that meaningful results will require that dependence is sufficiently weak overall
by limiting the scaling of $\mcC_N$,
which will be seen in the developed theory. 

\s

\begin{definition}
\label{def2}
Assume that $\mbP(B_{k,i} = 1) > 0$ 
for all $(k,i) \in \{0, 1, \ldots, p\} \times \{1, \ldots, M\}$
and define 
\beno
\Delta_{N} &\coloneqq& 
\dfrac{1}{M} \,
\dsum_{i=1}^{M} \,
\dsum_{k=0}^{p} \,
\mbP(B_{k,i} = 1) \,
\dsum_{j \in \{1, \ldots, M\} \setminus \{i\}} \,
\left(
\mbP(B_{k,j} = 1 \,|\, B_{k,i} = 1) - \mbP(B_{k,j} = 1)
\right). 
\ee
\end{definition}

\s

\begin{proposition}
\label{prop:1}
Let $\bB_i \coloneqq (B_{0,i}, B_{1,i}, \ldots, B_{p,i})$ ($i \in \{1, \ldots, M\}$)
be as defined in \eqref{eq:berns}
and define
\beno
\pi_{j\,|\,i}^{\bb_i}(\bb_j)
&\coloneqq& \mbP(\bB_j = \bb_j \,|\, \bB_i = \bb_i),
&& j \in \{1, \ldots, M\} \setminus \{i\}, \; i \in \{1, \ldots, M\}, \s \\
\pi_{j}(\bb_j)
&\coloneqq& \mbP(\bB_j = \bb_j),
&& j \in \{1, \ldots, M\}. 
\ee
Then 
\beno
\Delta_N
&\leq& \dfrac{1}{M} \,
\dsum_{i=1}^{M} \,
\dsum_{j \in \{1, \ldots, M\} \setminus \{i\}} \, \mbE \, d_{\tv}\left(\pi_{j\,|\,i}^{\bB_i}, \, \pi_{j} \right).
\ee
where $\Delta_N$ is as  defined in Definition \ref{def2}. 
\end{proposition}

Definition \ref{def2} presents an alternative method of quantifying the strength of dependence 
to that of Definition \ref{def1}, 
obtained by expressing the covariances $\cov(B_{k,i}, \, B_{k,j})$ as deviations of conditional and marginal probabilities. 
The deviations $\mbP(B_{k,j} = 1 \,|\, B_{k,i} = 1) - \mbP(B_{k,j} = 1)$ 
quantify how much the probability of events $B_{k,j} = 1$ can change 
with knowledge of the event $B_{k,i} = 1$. 
Proposition \ref{prop:1} shows that these deviations 
can be bounded above by expectations of total variation distances 
$\mbE \, d_{\tv}\left(\pi_{j\,|\,i}^{\bB_i}, \, \pi_{j} \right)$, 
quantifying the extent to which $\bB_i$ can influence each $\bB_j$ ($j \in \{1, \ldots, M\}$) 
through the expected total variation distance between the marginal distribution of $\bB_j$ 
and the conditional distribution of $\bB_j$ given $\bB_i$,
taking the expectation with respect to $\bB_i$.  
Similar to  Definition \ref{def1},
the influence of each $i \in \{1, \ldots, M\}$ measured in this way 
is averaged, 
allowing the bound in Proposition \ref{prop:1} to permit 
a small number of highly influential elements 
provided the overall dependence is sufficiently weak. 

\s

\begin{definition}
\label{def3}
Denote the conditional distribution of $B_{k,j}$
($j \in \{i+1, \ldots, M\}$)
conditioning on 
$(B_{k,1}, \ldots, B_{k,i})$ 
by 
\beno
\mbP_{i,j,k}^{\bb}(v)
&\coloneqq& \mbP\left(B_{k,j} = v \;|\; (B_{k,1}, \ldots, B_{k,i}) = \bb\right),
&& v \in \{0,1\}, \; \bb \in \{0,1\}^{i},
\ee
for each $i \in \{1, \ldots, M-1\}$ and $k \in \{0, 1, \ldots, p\}$, 
and define 
\beno
\mdelta &\coloneqq&
\max\limits_{(\bb,\bb^\prime) \in \{0, 1\}^{i} \times \{0,1\}^{i} \,:\, b_t = b_t^\prime, \, t < i} \;  
d_{\tv}\left(\mbP_{i,j,k}^{\bb}, \, \mbP_{i,j,k}^{\bb^\prime}\right), 
&&  i \in \{1, \ldots, M\}, \; k \in \{0, 1, \ldots p\},  
\ee
and $\dep \coloneqq \max\{\mcD_{N,0}, \, \mcD_{N,1}, \ldots, \mcD_{N,p}\}$,
where 
\beno
\mcD_{N,k}
&\coloneqq&
\dfrac{1}{M} \,
\dsum_{i=1}^{M} \, \left( 1 + \dsum_{j \in \{i+1, \ldots, M\}} \,
\mdelta
\right)^2,
&& k \in \{0, 1, \ldots, p\}.
\ee 
\end{definition}

\s 

Definition \ref{def3} is reminiscent of mixing conditions 
and other dependence quantifications in other approaches to establishing 
exponential concentration inequalities for dependent data \citep{Kontorovich2008},
with related approaches having already been applied in statistical network science applications 
\citep{SchweinbergerStewart2020}. 
In each of the above cases,
larger values of the quantities $\mcC$, 
$\Delta_N$,
and $\dep$
will result in weaker concentration,
as will be seen in Lemma \ref{lem:concentration}. 
We present an application in Section \ref{sec:application} for which each can be bounded, 
and discuss our weak dependence conditions in further detail in Section \ref{sec:weak_dep}. 
It is worth noting that even 
in the case when edges in the random graph are independent,
the events $\mG_{k,i}$ and $\mG_{k,j}$ can still be dependent.
A case in point is given by the degrees of nodes. 
Even if edges in the random graph are assumed to be independent, 
the degree of node $i \in \mN$ and $j \in \mN \setminus \{i\}$ 
are dependent, 
as both depend on the value of edge variable $X_{i,j}$. 
As a result, 
the collection of random variables
$B_{k,1}, \ldots, B_{k,M}$ ($k \in \{0, 1, \ldots, p\}$)
will, 
in general, 
be a collection of dependent random variables, 
even when edges in the random graph are independent. 

\s

\begin{lemma}
\label{lem:concentration}
Consider a simple random graph $\bX$ and let $\empnox : \mbX \mapsto [0,1]^{p+1}$
be as defined in \eqref{eq:emp}.
Then,
for all $t > 0$,
\be
\label{conc1}
\mbP\left( \norm{\emp - \truth}_{\infty} \,\geq\, t \right)
&\leq& 2 \, \exp\left( - \dfrac{2 M  t^2}{\dep} + \log(1+p) \right), 
\ee
where $\dep$ is defined in Definition \ref{def3},
and 
\be
\label{conc2}
\mbP\left( \norm{\emp - \truth}_{\infty} \,\geq\, t \right)
&\leq& \dfrac{1 +\min\left\{\mcC_N, \; \Delta_N\right\}}{M \, t^2}, 
\ee
where $\mcC_N$ is defined in Definition \ref{def1} and 
$\Delta_N$ is defined in Definition \ref{def2}. 
\end{lemma}

We will leverage Lemma \ref{lem:concentration} to establish the statistical theory of this work. 
The exponential inequality in Lemma \ref{lem:concentration} 
is most suitable for demonstrating the almost sure convergence results presented in 
Corollaries \ref{cor:1} and \ref{cor:2}. 
First, 
in the proofs of coming theoretical results, 
we will utilize union bounds which will render weaker concentration inequalities insufficient for certain results. 
While there are two related inequalities due to \citet{Kontorovich2008} and \citet{Chazottes2007},
as well as inequalities utilizing the Dobrushin's uniqueness condition \citep{Dobruschin1968, Dagan2019}, 
these approaches will not lead to suitable inequalities for this work. 
Second, 
the weaker concentration inequality of Lemma \ref{lem:concentration} 
given in \eqref{conc2} presents bounds in terms of quantities which may be more interpretable,
that of covariances and deviations of marginal and conditional probabilities.  
We develop the results of this work,  
which are non-parametric, 
under the three different assumptions on the dependence of the events of interest within the random graph
that were presented in Definitions \ref{def1}, \ref{def2}, and \ref{def3}.
The applicability of the theory developed in this work rests on whether 
these assumptions are reasonable for the specific application considered.

\subsection{Non-asymptotic high probability bounds on the $\ell_{\infty}$-error of empirical distributions\label{sec:main_res}}

We derive uniform bounds on the error of empirical distributions 
of sequences of graph statistics in Theorem \ref{thm:main1}, 
in settings where the edge variables can be dependent  
and covering a broad range of sequences of graph statistics. 
These uniform bounds facilitate deriving upper bounds on rates of convergence.

\begin{theorem} 
\label{thm:main1}
Consider a simple random graph $\bX$
and let  
$\empnox : \mbX \mapsto [0,1]^{p+1}$ be as defined in \eqref{eq:emp}. 
Then
\beno
\mbP\left( \norm{\emp - \truth}_{\infty} \,<\, \sqrt{\dfrac{3}{2}} \,   
\sqrt{\dfrac{\dep \, \log(\max\{M, \, 1+p\})}{M}} \right)
&\geq& 1  - \dfrac{2}{\max\{M, \, 1+p\}^2}, 
\ee
where $\dep$ is defined in Definition \ref{def3},
and,
for all significance levels $\alpha \in (0, 1)$, 
\beno
\mbP\left(  \norm{\emp - \truth}_{\infty} \,<\, 
\sqrt{\dfrac{1 + \min\left\{|\mcC_N|, \; |\Delta_N|\right\}}{\alpha \, M}} \right)
&\geq& 1 - \alpha,
\ee
where $\mcC_N$ is defined in Definition \ref{def1} and $\Delta_N$ is defined in Definition \ref{def2}. 
\end{theorem}

The quantities 
$\mcC_N$, $\Delta_N$, and $\dep$ 
place certain restrictions on the scope of what sequences of graph statistics can be chosen,
as these quantities cannot grow too quickly relative to $M$,  
otherwise consistency will not be established. 
While each represents a different quantification of the dependence in the random graph, 
each essentially requires that the sequences of graph statistics considered 
cannot produce events which are too strongly dependent. 
With regards to the definition of $\dep$ in Definition \ref{def3}, 
we assume that we bound the total variation distances of the conditional probabilities distributions with probability $1$. 
We can prove a result analogous to Theorem \ref{thm:main1} which
weakens this assumption,
allowing our results to cover a larger scope of sequences of graph statistics and random graphs.

\begin{theorem}
\label{thm:main2}
Consider a simple random graph $\bX$ and let $\empnox : \mbX \mapsto [0,1]^{p+1}$
be as defined in \eqref{eq:emp}.
Assume there exists a subset $\mbX_0 \subseteq \mbX$,
a constant $N_0 \geq 3$, and 
a function $r : \{3, 4, \ldots\} \mapsto (0, 1)$ such that,
for each $j \in \{1, \ldots, M\}$, 
\be
\label{eq:delta_r}
\mdelta(\mbX_0) &\coloneqq&
\max\limits_{(\bb,\bb^\prime) \in \mbB_{k,i}(\mbX_0) \times \mbB_{k,i}(\mbX_0) \,:\, b_t = b_t^\prime, \, t < i} \;
d_{\tv}\left(\mbP_{i,j,k}^{\bb}, \, \mbP_{i,j,k}^{\bb^\prime}\right)
&&  i \in \{1, \ldots, M\}, \; k \in \{0, 1, \ldots p\},
\ee 
where 
$\mbB_{k,i}(\mbX_0)$ is the subset of $\bb \in \{0,1\}^{i}$ 
($ i \in \{1, \ldots, M\}$) 
for which there exists $\bx \in \mbX_0$ for which $B_{k,t} = b_t$ ($t \leq i$), 
\be
\label{eq:RN}
r(N)
&\leq& \sqrt{\dfrac{\dep(\mbX_0) \, \log(\max\{M, 1+p\})}{M}}, 
\ee
with $\dep(\mbX_0)$ defined as in Definition \ref{def3} except using the definition of $\mdelta(\mbX_0)$ in \eqref{eq:delta_r},
and assume that  
\be
\label{eq:A2-1}
\mbP(\bX \in \mbX_0) 
&\geq& 1 - r(N), 
&& 
N \geq N_0. 
\ee
Then,
under assumptions \eqref{eq:delta_r}, \eqref{eq:RN}, and \eqref{eq:A2-1},
we have  
\beno
\mbP\left( \norm{\emp - \truth}_{\infty} \,<\,
\sqrt{\dfrac{27}{2}} \, 
\sqrt{\dfrac{\dep(\mbX_0) \, \log(\max\{M, 1+p\})}{M}}
\right)
&\geq& 1 - r(N) - \dfrac{4}{\max\{M, 1+p\}^2}. 
\ee
\end{theorem}

\s

Theorem \ref{thm:main2} extends the results of Theorem \ref{thm:main1} 
to settings where certain configurations of the network $\bX \in \mbX$ may give rise 
to large total variation distances in the definition of $\mdelta$ in Definition \ref{def3},
which in turn would give rise to larger values of $\dep$ as defined in Definition \ref{def3}.   
A case in point is again given by the degree distribution.  
Suppose we condition on the node degrees of nodes $i \in \{1, 2, \ldots, N-1\}$,
i.e.,
all but one node degree,
and consider the conditional probability distributions of $B_{0,N}, B_{1,N}, \ldots, B_{N-1,N}$
given the node degrees of nodes $i \in \{1, 2, \ldots, N-1\}$. 
Taking the maximum over all possible graph configurations, 
we can induce a total variation distance of $1$ 
by considering the complete and empty graphs as the conditioning graphs,
which will force $B_{0,N}$ and $B_{N-1,N}$ to be either $0$ or $1$ (almost surely), 
depending on the conditioning event in the definition of $\delta_{N-1,N,0}$ 
(for $B_{0,N}$)
and $\delta_{N-1,N,N-1}$
(for $B_{N-1,N}$)
in Definition \ref{def3}. 
However, 
the cases where a single network has either no edges or all edges would be unlikely events for most applications and models. 
This is where Theorem \ref{thm:main2} innovates upon Theorem \ref{thm:main1}.
Under the setup of Theorem \ref{thm:main2},
we can circumvent pathological cases such as the example above
which occur with low probability by restricting the definition of $\mdelta(\mbX_0)$ to only subsets $\mbX_0 \subset \mbX$
which both occur with high probability and for which the total variation distances defining $\mdelta(\mbX_0)$
are not too large to render the results of our statistical theory meaningless.

\hide{
Consider the scenario where $m = 1$ and focus on the conditional probability distribution of node $2 \in \mN$ 
given all other nodes $j \in \mN \setminus \{1, 2\}$, 
where 
\beno
B_{d,i} 
&\coloneqq& \one\left(\, \dsum_{j \in \mN \setminus \{i\}} \, X_{i,j} \,=\, d \right),
&& i \in \mN, 
& d \in \{0, 1, \ldots, N-1\}.  
\ee 
Define $(\bb, \bb^\prime) \in \{0,1\}^N \times \{0,1\}^N$ as follows:
\ben
\item Set $b_j^\prime = b_j$ for all $j \in \mN \setminus \{1\}$ for any value $b_j \in \{0, 1\}$, and 
\item Define $b_1 = 0$ and $b_1^\prime = 1$. 
\een 
In this case, 
\beno
d_{\tv}(\mbP_{0,2}^{\bb}, \, \mbP_{0,2}^{\bb^\prime})
\= \dfrac{1}{2} \, \dsum_{d=0}^{1} \, \left| \mbP_{0,2}^{\bb}(d) - \mbP_{0,2}^{\bb^\prime}(d) \right|
\= 1,
\ee
because $\mbP_{0,2}^{\bb}(0) = 1$ and $\mbP_{0,2}^{\bb^\prime}(1) = 1$,
owing to the construction of the tuple $(\bb, \bb^\prime)$ above. 
In words, 
this occurs because if nodes $1, 3, \ldots, N$ all have degree $0$,
which is the event 
\beno
B_{0,i} \= 
\one\left( \, \dsum_{j \in \mN \setminus \{i\}} \, X_{i,j} \,=\, 0 \right)
\= 0,
&& i \in \{1, 3, \ldots, N\},  
\ee 
then node $2$ must have degree equal to $0$;  
and conversely, 
if node $1$ has degree greater than $0$, 
i.e.,
\beno
B_{0,1} 
\= \one\left( \, \dsum_{j \in \mN \setminus \{i\}} \, X_{i,j} \,=\, 0 \right)
\= 0,
\ee
and nodes $3, \ldots, N$ all have degree $0$,
then node $2$ cannot have have degree equal to $0$, 
because there must be some node connected to node $1$ if $B_{0,1} = 0$. 
However, 
the case where a single network has only a single edge will be an unlikely event for most models 
and applications of interest. 
This is where Theorem \ref{thm:main2} innovates upon Theorem \ref{thm:main1}. 
Under the setup of Theorem \ref{thm:main2},
we can circumvent pathological cases such as the example above 
which occur with low probability by restricting the definition of $\mdelta$ to only subsets $\mbX_0 \subset \mbX$ 
which occur with high probability and for which the total variation distances defining $\mdelta$ 
are not too large to render the results of our statistical theory meaningless.  
This allows our results to extend to a much greater scope of networks and sequences of graph statistics. 
}

\subsection{Applications to degree distributions} 
\label{sec:deg_app}

We next prove a corollary to Theorem \ref{thm:main2} 
for the empirical degree distribution, 
which was given as an example in Section \ref{sec:1} and is defined in \eqref{eq:deg_ex}. 
The degree distribution is one of the most fundamental aspects of a network,
and the importance of this result lies in the fact that often practitioners of 
statistical network science
rely on information and insights gained through the empirical degree distributions. 
Corollary \ref{cor:1} provides rigorous statistical foundations 
for drawing inferences from empirical degree distributions of networks,
covering a broad range of settings that notably include networks with dependent edge variables. 
While Corollary \ref{cor:1}
is stated for degree distributions of undirected random graphs, 
it is straightforward to extend the results to directed random graphs, 
covering either total degree distributions, 
as well as out-degree and in-degree distributions. 
We do not present this extension, 
but note that we appeal to this result in our application to a directed network in Section \ref{sec:application}. 

To lay the foundation for Corollary \ref{cor:1}, 
we will let $M = N$ and $s_d(\bx)$ ($d \in \{0, 1, \ldots, N-1\}$)
be as defined in \eqref{eq:deg_ex}
and define $\bB_d \coloneqq (B_{d,1}, \ldots, B_{d,N})$ ($d \in \{0, 1, \ldots, N-1\}$) 
by defining each  
\beno
B_{d,i} 
&\coloneqq&  
\one\left( \, \dsum_{j \in \mN \setminus \{i\}} \, X_{i,j} \,=\, d \right),
&& d \in \{0, 1, \ldots, N-1\}, \;\; i \in \mN.
\ee
Following the approach of Theorem \ref{thm:main2}, 
we assume there exist a subset $\mbX_0 \subseteq \mbX$ and constant $N_0 \geq 3$ such that 
\be
\label{eq:cor1_assumption_a}
\mbP(\bX \in \mbX_0)
&\geq& 1 - \dfrac{2}{N^2},
&& N \geq N_0,  
\ee
and there exists,  
for each $m \in \mN$, 
a subset $\mM_m \subset \mN$ 
and constants $\alpha_{m,i} \in [0, \infty)$ ($i \in \mN \setminus (\mM_m \cup \{m\}$)
such that 
\be
\label{eq:cor1_assumption_b}
\max\limits_{(\bb,\bb^\prime) \in \mbB_{d,m}(\mbX_0) \times \mbB_{d,m}(\mbX_0) \,:\, b_t = b_t^\prime, \, t < m} \;
d_{\tv}\left(\mbP_{m,i,d}^{\bb}, \, \mbP_{m,i,d}^{\bb^\prime}\right)
\;\leq\; \alpha_{m,i}, 
& 
i \in \mN \setminus (\mM_m \cup \{m\}), 
\;d \in \{0, 1, \ldots, N-1\},  
\ee
where 
$\mbB_{d,i}(\mbX_0)$ is the subset of $\bb \in \{0,1\}^{i}$
($i \in \{1, \ldots, M\}$)
for which there exists $\bx \in \mbX_0$ for which $B_{d,t} = b_t$ ($t \leq i$).
As a result of this assumption,  
for each $m \in \mN$ and $d \in \{0, 1, \ldots, N-1\}$, 
\beno
1 + \dsum_{i \in \mN \setminus \{m\}} \, \delta_{m,i,d}(\mbX_0)
&\leq& 1 + M_m + \dsum_{i \in \mN \setminus (\mM_m \,\cup \, \{m\})} \, \alpha_{m,i}, 
\ee
defining $M_m \coloneqq |\mM_m|$ ($m \in \mN$) and 
noting that $d_{\tv}\left(\mbP_{m,i,d}^{\bb}, \; \mbP_{m,i,d}^{\bb^\prime} \right) \leq 1$,
which in turn implies that 
\be
\label{eq:cor1_dep_bound}
\dep(\mbX_0) 
&\coloneqq& \max\limits_{d \in \{0, 1, \ldots, N-1\}} \, \left[ 
\dfrac{1}{M} \,
\dsum_{m=1}^{M} \, \left( 1 + \dsum_{i \in \{m+1, \ldots, M\} \setminus \{m\}} \,
\delta_{m,i,d}(\mbX_0)
\right)^2 \right]
&\leq& (1 + M_{\max} + \alpha_{\max})^2,
\ee
defining $M_{\max} \coloneqq \max\{M_1, \ldots, M_N\}$ and 
$\alpha_{\max} = \max_{m \in \mN} \sum_{i \in \mN \setminus (\mM_m \,\cup\, \{m\})}  \alpha_{m,i}$.

The assumption of both \eqref{eq:cor1_assumption_a} and \eqref{eq:cor1_assumption_b},
which leads to the bound on $\dep(\mbX_0)$ given in \eqref{eq:cor1_dep_bound},
is comparable to strong mixing conditions 
(i.e., it is reminiscent of the $\alpha$-mixing condition) \citep{Bradley2005},
placed under a high-probability condition.  
Bounding $\dep$ and $\dep(\mbX_0)$ becomes straightforward under independence assumptions placed on the degrees of nodes,
but as mentioned already, 
this assumption would be unreasonable because degrees are not independent.  
A more realistic assumption would be a form of M-dependence similar to the local dependence assumption 
of local dependence random graph models \citep{Schweinberger2015,SchweinbergerStewart2020},
which we utilize in the application presented in Section \ref{sec:application}. 
However, 
even this assumption may be too strong for certain applications. 
The conditions outlined in \eqref{eq:cor1_assumption_a} and \eqref{eq:cor1_assumption_b}, 
which lead to the bound given in \eqref{eq:cor1_dep_bound}, 
represent a compromise between rich local dependence and weak global dependence,
by incorporating a form of strong local dependence 
(controlled via $M_{\max}$)
and weak global dependence reminiscent of strong mixing conditions (controlled via $\alpha_{\max}$). 
In practical terms, 
this assumption allows arbitrarily strong influence of the degree of a node $m \in \mN$ 
on other nodes 
$j \in \mM_m$, 
but permits only weak influence of the degrees of nodes $i \in \mN \setminus (\mM_m \cup \{m\})$.

\begin{corollary}
\label{cor:1}
Under the assumptions of Theorem \ref{thm:main2}
with $s_k(\bx)$ given by \eqref{eq:deg_ex} 
and the assumption of both \eqref{eq:cor1_assumption_a} and \eqref{eq:cor1_assumption_b},
there exists a constant $N_0 \geq 3$ such that 
\beno
\mbP\left( \norm{\emp - \truth}_{\infty} \,<\, (1 + M_{\max} + \alpha_{\max}) \, 
\sqrt{\dfrac{3}{2}} \, \sqrt{\dfrac{\log(N)}{N}} \right)
&\geq& 1 - \dfrac{6}{N^2},
&& N \geq N_0.  
\ee 
Assuming 
$M_{\max} + \alpha_{\max} = o\left(\hspace{-.1cm}\sqrt{N / \log(N)} \hspace{.05cm}\right)$, 
the error $\norm{\emp - \truth}_{\infty}$ converges almost surely to $0$ as $N \to \infty$.   
\end{corollary}

Corollary \ref{cor:1} demonstrates uniform convergence of the empirical degree distribution as $N \to \infty$, 
provided the dependence among the degrees is not globally strong, 
in the sense that we require the quantities $M_{\max}$ and $\alpha_{\max}$ 
to satisfy 
$M_{\max} + \alpha_{\max} = o\left(\hspace{-.1cm}\sqrt{N / \log(N)} \hspace{.05cm}\right)$. 
Of note, 
neither the sparsity of the random graph nor the heterogeneity of node degrees 
affect our consistency theory. 
This means that a network can demonstrate marked heterogeneity among the node degrees
and the empirical degree distribution $\emp$ will be asymptotically stable in the sense that 
it convergences uniformly to $\truth$,
again provided the weak dependence criterion of the corollary is satisfied. 

Lastly, 
we prove a second corollary to Theorem \ref{thm:main1} under a sparse inhomogeneous Bernoulli random graph model. 

\begin{corollary}
\label{cor:bern}
Consider an inhomogeneous Bernoulli random graph $\bX$,
i.e., 
edge variables $X_{i,j}$ ($\{i,j\} \subset \mN$) are assumed 
to be 
independent Bernoulli random variables   
with heterogeneous probabilities $\mbP(X_{i,j} = 1) \in (0, 1)$ ($\{i,j\} \subset \mN$).
Then 
\beno
\Delta_N
&\leq& \dfrac{2}{N} \dsum_{i=1}^{N} \mbE \, d_i(\bX),
&& \mbP\left(  \norm{\emp - \truth}_{\infty} \,<\,
\sqrt{\dfrac{1 + (2 \,/\, N) \sum_{i=1}^{N} \,  \mbE \, d_i(\bX)}{\alpha \, N}} \, \right)
&\geq& 1 - \alpha,
\ee
for all $\alpha \in (0, 1)$, 
where $d_i(\bX)$ is the degree of node $i \in \mN$ and 
$\Delta_N$ is defined in Definition \ref{def2}. 
\end{corollary}

It is worth noting that that Corollary \ref{cor:bern} covers latent space models which assume that edge 
variables are conditionally independent given the latent variables, 
as such models can be represented as inhomogeneous Bernoulli random graph models. 
This corollary helps to highlight a key fact concerning the rates of convergence of empirical degree distributions.  
Sparsity can help to mitigate the influence of dependence among degrees, 
as the quantity $\Delta_N$ can be bounded above by the average expected degree of nodes. 
If the expected degrees of nodes is bounded above, 
then $\Delta_N$ will likewise be bounded above. 
More generally, 
we can establish consistent estimation provided the expected degrees of nodes grows slower than $N$,
in the sense that $\norm{\emp - \truth}_{\infty}$ will converge in probability to $0$.

\subsection{Applications to edgewise shared partner distributions} 
\label{sec:esp_app}

The study of transitivity and dependence in network data applications dates back to at least the 1970s
with work by 
\citet{Holland1972},
and modeling expressions of network transitivity through triangle counts was a key motivation 
in the seminal work of \citet{Frank1986},
and has been a focus in the exponential-family random graph model literature 
\citep{Lusher2012}. 
In general, 
we can understand the network phenomena of transitivity within a statistical context 
as the change in the conditional probability of an edge due to the presence or absence of common connections to 
other nodes.
Colloquially, 
this may be described in the context of social network analysis as 
{\it the friend of my friend is my friend},
where it is common to observe positive transitivity, 
meaning two nodes are more likely to be connected if they have a common connection to at least one other node in the network,
compared with the case when they have no other connections. 
A more modern model of network transitivity is given by curved exponential family parameterizations 
of edgewise shared partner distributions 
\citep{Hunter2006,Robins2007}. 
Additionally, 
the edgewise shared partner distribution is frequently 
included in goodness-of-fit diagnostics for evaluating the fit of estimated models \citep{Hunter2008}. 
Notably, 
the {\tt R} package {\tt ergm} includes the edgewise shared partner distribution as a default 
in its goodness-of-fit diagnostics
\citep{Krivitsky2023}.  
See \citet{Stewart2019} for an in-depth analysis and discussion of transitivity in the context of social network analysis, 
as well as for a review of the edgewise shared partner distribution and relevant models and literature.  

\begin{definition}
\label{def:esp}
The {\it edgewise shared partner distribution} 
is defined to be a sequence of graph statistics $\bs(\bx)$  
given by 
\be
\label{eq:esp}
s_{k}(\bx)
\= \dsum_{\{i,j\} \subset \mN} \, 
x_{i,j} \, \one\left(  \dsum_{h \in \mN \setminus \{i,j\}} \, x_{i,h} \, x_{j,h} \,=\,k \right),
&& k \in \{0, 1, \ldots, N-2\}, 
\ee
where each $s_{k}(\bx)$ counts the number of connected pairs of nodes $\{i,j\} \subset \mN$ 
with exactly $k$ other mutual connections. 
For a given pair of nodes $\{i,j\} \subset \mN$, 
the nodes $h \in \mN \setminus \{i,j\}$ with both $x_{i,h} = 1$ and $x_{j,h} = 1$,
i.e.,
both $i$ and $j$ are mutually connected to $h$ in $\bx$,   
are called the {\it shared partners} of $i$ and $j$.   
\end{definition}

\hide{
In words,
each summand (for a given $k \in \{0, 1, \ldots, N-2\}$) 
in \eqref{eq:esp} is an indicator random variable indicating whether 
\ben
\item The edge $x_{i,j}$ between nodes $i \in \mN$ and $j \in \mN$ is present in the network $\bx$, and 
\item Nodes $i \in \mN$ and $j \in \mN$ each have precisely $k \in \{0, 1, \ldots, N-2\}$ 
connections to common nodes $h \in \mN \setminus \{i,j\}$,
which are called the {\it shared partners} due to the fact that nodes $i \in \mN$ and $j \in \mN$ 
share common connections to these $k \in \{0, 1, \ldots, N-2\}$ nodes $h \in \mN \setminus \{i,j\}$. 
\een}

For this example, 
we define 
define $\emp$ in \eqref{eq:emp} to be 
\be
\label{esp_dist}
\empk
\= \dfrac{s_{k}(\bX)}{\norm{\bX}_1}, 
&& k \in \{0, 1, \ldots, N-2\},  
\ee
where 
$s_{k}(\bX)$ ($k \in \{0, 1, \ldots, N-2\}$) 
are the edgewise shared partner statistics defined in Definition \ref{def:esp} and 
where $\norm{\bX}_1 = \sum_{\{i,j\} \subset \mN} \, X_{i,j}$ 
is the edge count of the network $\bX$,
noting that the sum in \eqref{eq:esp} is a sum of $\norm{\bX}_1$ terms. 
Distinct from the previous example which was the degree distribution, 
the value of $M$ (which in this example will be $\norm{\bX}_1$) 
is not a deterministic constant, 
but is in fact a random quantity. 
In order to overcome this, 
we will utilize Theorem \ref{thm:main2} and incorporate into the high-probability set $\mbX_0$ 
a condition which allows us to bound $\norm{\bX}_1$,
effectively bounding $M$ in the worst case for a given assumption on the density of the network.  

We control the quantity $\dep(\mbX_0)$ in similar fashion to Corollary \ref{cor:1}.
Let 
$\mE \coloneqq \big\{\{i,j\} : (i,j) \in \mN \times \mN \mbox{ with } i < j \big\}$
be the set of all unordered pairs of nodes in the network $\bX$.  
Define a bijective map $\xi : \{1, \ldots, \binom{N}{2}\} \mapsto \mE$,
which essentially constructs an arbitrary ordering 
of the pairs of edges which are enumerated in $\mE$, 
and define $\bB_k \coloneqq (B_{k,1}, \ldots, B_{k,M})$ ($k \in \{0, 1, \ldots, N-2\}$) 
by defining,
for each $q \in \{1, \ldots, \binom{N}{2}\}$ with $\{i,j\} = \xi(q)$, 
\beno
B_{k,q} 
&\coloneqq&  
\one\left( \, \dsum_{h \in \mN \setminus \{i,j\}} \, X_{i,h} \, X_{j,h} \,=\, k \right),
&& k \in \{0, 1, \ldots, N-2\}. 
\ee
In words, 
for each $q \in \{1, \ldots, \binom{N}{2}\}$ with $\{i,j\} = \xi(q)$,
$B_{k,q} = 1$ if and only if nodes $i$ and $j$ have precisely $k$ shared partners. 
Assume that there exist a subset $\mbX_0 \subseteq \mbX$ and constant $N_0 \geq 1$ with the property that 
\be
\label{eq:cor2_assumption_a}
\mbP(\bX \in \mbX_0)
&\geq& 1 - \dfrac{2}{N^2},
&& N \geq N_0,  
\ee
and,  
for each $m \in \{1, \ldots, \binom{N}{2}\}$, 
there exist a subset $\mM_m \subset \{1, \ldots, \binom{N}{2}\}$ 
and non-negative constants $\alpha_{m,q} \in [0, \infty)$ 
($q \in \{1, \ldots, \binom{N}{2}\} \setminus (\mM_m \cup \{m\})$) 
such that,
for each $q \in \{1, \ldots, \binom{N}{2}\} \setminus (\mM_m \cup \{m\})$,
we have the bounds  
\be
\label{eq:cor2_assumption_b}
 \max\limits_{(\bb,\bb^\prime) \in \mbB_{k,i}(\mbX_0) \times \mbB_{k,i}(\mbX_0) \,:\, b_t = b_t^\prime, \, t < m} 
\;\; 
d_{\tv}\left(\mbP_{m,q,k}^{\bb}, \; \mbP_{m,q,k}^{\bb^\prime} \right)
&\leq& \alpha_{m,q}, 
&& k \in \{0, 1, \ldots, N-2\},  
\ee
where 
$\mbB_{k,i}(\mbX_0)$ is the subset of $\bb \in \{0,1\}^{i}$
($i \in \{1, \ldots, M\}$)
for which there exists $\bx \in \mbX_0$ for which $B_{k,t} = b_t$ ($t \leq i$).
As a result of this assumption,  
for each $m \in \{1, \ldots, \binom{N}{2}\}$ and each $k \in \{0, 1, \ldots, N-2\}$, 
\beno
1 + \dsum_{q \in \left\{1, \ldots, \binom{N}{2}\right\} \setminus \{m\}} \, \delta_{m,q,k}(\mbX_0)
&\leq& 1 + M_m + \dsum_{q \in \left\{1, \ldots, \binom{N}{2}\right\} \setminus (\mM_m \,\cup\, \{m\})} \, \alpha_{m,q}, 
\ee
defining $M_m \coloneqq |\mM_m|$ ($m \in \{1, \ldots, \binom{N}{2}\}$) and 
noting that $d_{\tv}\left(\mbP_{m,q,k}^{\bb}, \; \mbP_{m,q,k}^{\bb^\prime} \right) \leq 1$,
which in turn implies that 
\be
\label{eq:cor2_dep_bound}
\dep(\mbX_0) 
&\coloneqq& \max\limits_{k \in \{0, 1, \ldots, N-2\}} \, \left[ 
\dfrac{1}{M} \,
\dsum_{m=1}^{M} \, \left( 1 + \dsum_{q \in  \left\{m+1, \ldots, \binom{N}{2}\right\}} \, 
\delta_{m,q,k}(\mbX_0) \right)^2 \right]
&\leq& (1 + M_{\max} + \alpha_{\max})^2,
\ee
defining $M_{\max} \coloneqq  \{M_1, \ldots, M_{\binom{N}{2}}\}$
and 
$\alpha_{\max} \coloneqq \max_{m \in \left\{1, \ldots, \binom{N}{2}\right\}} \, 
\sum_{q \in \left\{1, \ldots, \binom{N}{2}\right\} \setminus (\mM_m \,\cup\, \{m\})} \,  \alpha_{m,q}$, 
similarly as in Corollary \ref{cor:1}.

\begin{corollary}
\label{cor:2}
Under the assumptions of Theorem \ref{thm:main2}
with 
$s_k(\bx)$ as given in \eqref{eq:esp} 
and assuming both \eqref{eq:cor2_assumption_a} and \eqref{eq:cor2_assumption_b} and 
\be
\label{deg_control}
\mbP\left( \bX \in \left\{\bx \in \mbX \,:\, \norm{\bx}_1 \geq N^{\beta} \right\} \;\cap\; \mbX_0 \right)
&\geq& 1 - \dfrac{2}{N^2}, 
\ee
for some $\beta > 0$ and $\mbX_0$ given by \eqref{eq:cor2_assumption_a} and \eqref{eq:cor2_assumption_b}, 
there exists a constant $N_0 \geq 3$ such that 
\beno
\mbP\left( \norm{\emp - \truth}_{\infty} \,<\,
(1 + M_{\max} + \alpha_{\max}) \,
\sqrt{\dfrac{3}{2}} \,
\sqrt{\dfrac{\log(N)}{N^{\beta}}} \right)
&\geq& 1 - \dfrac{11}{N^2},
&& N \geq N_0. 
\ee 
Assuming $M_{\max} + \alpha_{\max} = o\left(\hspace{-.1cm}\sqrt{N^{\beta} / \log(N)} \hspace{.05cm}\right)$,
the error $\norm{\emp - \truth}_{\infty}$ converges almost surely to $0$ as $N \to \infty$.   
\end{corollary}

A remark is in order regarding the role of $\beta > 0$ in Corollary \ref{cor:2}. 
As mentioned, 
the value of $M$ in this example is not a deterministic quantity, 
meaning that the denominator of the proportions $\emp$ will be random as a proportion 
of the subset of edges present in the graph.  
To overcome this, 
we incorporate a minimum density assumption through the condition in \eqref{deg_control},
which allows us to bound the value of $\norm{\bX}_1$ with probability $1$ 
under the conditioning event of \eqref{deg_control}.  
The minimal scaling of the density of the network through the specification of $\beta > 0$
then effectively determines the value of $M$ (in the worst case), 
with lower values of $\beta > 0$ resulting in slower rates of convergence.  

\subsection{Discussion of the weak dependence conditions} 
\label{sec:weak_dep}

The main assumption of this work lies in the reasonableness of the assumption
that the quantities $\mcC_N$, $\Delta_N$, and $\dep$,
which were defined in Definitions \ref{def1}, \ref{def2}, and \ref{def3},
respectively, 
will be bounded or grow slowly relative to the quantity $M$ in our theory.  
We argue that this assumption will be reasonable in many applications, 
especially those in the social and life sciences,
and discuss why our conditions are relevant to many real world applications.

For the sake of argument, 
consider a social network in the form of a friendship network and consider two randomly selected individuals on two different 
continents. 
We might ask the question:
If person A makes a new friend, how can we expect that to influence the friendships of person B?
In most sociological settings, 
is reasonable to assume that individuals are able to influence only their local neighborhood, 
giving rise to potentially strong local dependence,
but weak global dependence. 
Recent works in the statistical network analysis literature have demonstrated the importance
of localized dependence in network data applications on both statistical grounds and scientific grounds
\citep{Schweinberger2015,Stewart2019,SchweinbergerStewart2020,Stewart2020}.
These works were inspired by the relevance of this type of assumption to real world network data applications.
It is from these considerations that 
we argue that the weak dependence assumptions of this work are highly applicable to real world networks.

These ideas are central to the definitions of $\mcC_N$, $\Delta_N$, and $\dep$ in Definitions \ref{def1}, \ref{def2}, and \ref{def3},
and motivated the construction of the bounds on $\dep$  presented in Corollaries \ref{cor:1} and \ref{cor:2},
where we allowed for a small (relative to the size of the network) neighborhood of arbitrarily strong dependence, 
but placed a weak dependence assumption on the quantities outside of that neighborhood;  
these were controlled by the values of $M_{\max}$ and $\alpha_{\max}$ in Corollaries \ref{cor:1} and \ref{cor:2}. 
With regards to $\mcC_N$,
which quantified the dependence among the events of interest through covariances,
this condition requires the correlation between events to not be too large,
whereas $\Delta_N$ requires that the marginal distributions of the event statistics $B_{0,i}, B_{1,i}, \ldots, B_{p,i}$
not be too sensitive to conditional events for other event statistics $B_{0,j}, B_{1,j}, \ldots, B_{p,j}$. 
With respect to the running example of the degree distribution, 
this condition restricts  the extent to which the degree of node $i \in \mN$ 
can change the marginal distributions of the degrees of nodes $j \in \mN \setminus \{i\}$,
as measured by the total variation distance between the marginal distribution and the conditional distribution,
as presented in Proposition \ref{prop:1}.
This is analogous to the discussion above concerning local versus global dependence, 
where an individual node may be able to influence a small (relative to the entire network) number of nodes,
but is unable or unlikely to be able to exert strong global influence over the network.

\section{Empirical results\label{sec:3}}

\subsection{Simulation study 1: curved exponential-family random graph models} 

The first simulation study we conduct explores curved exponential parameterizations 
of exponential-family random graph models. 
Curved exponential parameterizations for exponential-family random graph models 
date back to \citet{Snijders2006} and \citet{Hunter2007},
and have since been shown to remediate issues with early attempts at constructing models of edge dependent 
(e.g., model degeneracy) \citep{Hunter2006,Stewart2019,SchweinbergerStewart2020}. 
The most prominent example of curved exponential parameterizations in the literature 
includes geometrically-weighted model terms, 
which parsimoniously parameterize sequences of graph statistics, 
typically edge-wise shared partner distributions or degree distributions. 

In this simulation study, 
we demonstrate the results of Theorems \ref{thm:main1} and \ref{thm:main2} 
in the context of a curved exponential-family random graph model which includes two model terms: 
the edge count statistic and the geometrically-weighted edgewise shared partner statistics. 
We can write down the joint distribution for $\bX$ in this simulation study as 
\be
\label{eq:sim1}
\mbP(\bX = \bx) 
&\propto& \exp\left( \theta_1 \dsum_{\{i,j\} \subset \mN} \, x_{i,j} 
+ \dsum_{k=1}^{N-2} \, \eta_{k}(\theta_2, \theta_3) \, \dsum_{\{i,j\} \subset \mN} x_{i,j} \, 
\one\left(\dsum_{h \in \mN \setminus \{i,j\}} \, x_{i,h} \, x_{j,h} \,=\, k \right) \right), 
\ee 
where $(\theta_1, \theta_2, \theta_3) \in \mbR \times \mbR \times [0, \infty)$ and
$\eta_{k}(\theta_2, \theta_3) \coloneqq \theta_2 \, \exp(\theta_3) \,[ 1 - (1 - \exp(\theta_3))^k]$
($k \in \{1, \ldots, N-2\}$),
inducing dependence among edges,
as the distribution will not factorize with respect to the edge variables when $\theta_2 \neq 0$. 

\begin{figure}[t]
\centering 
\includegraphics[scale = .6]{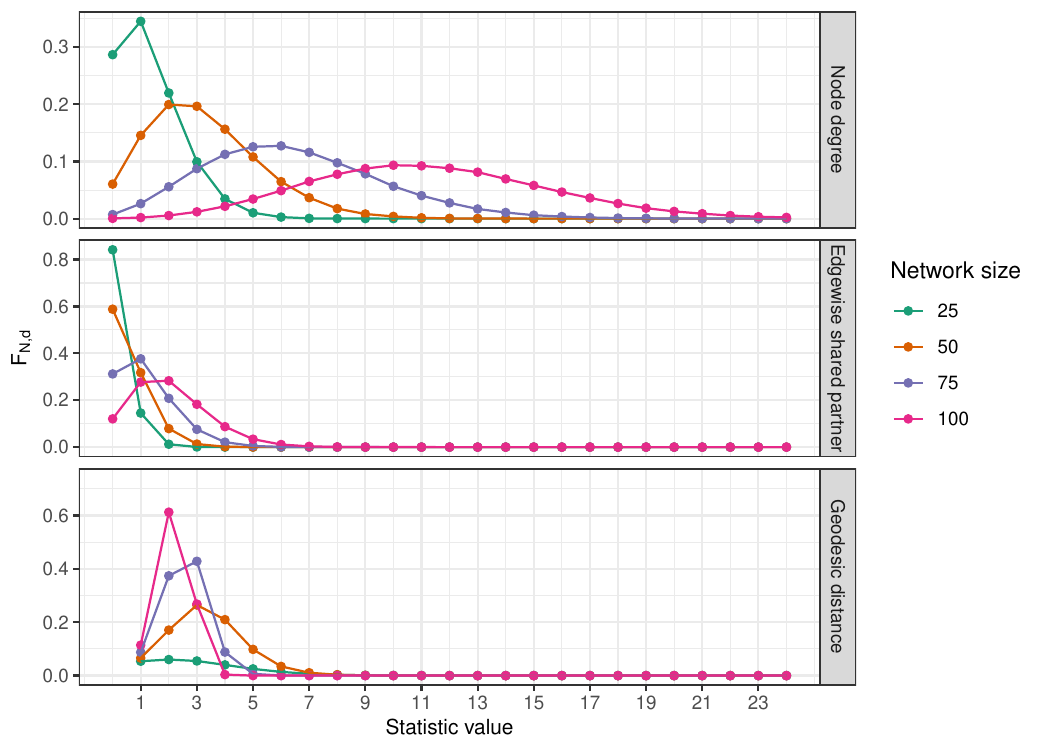} %
\caption{\label{fig:sim1-1} Results of simulation study 1. 
Estimated theoretical marginal distributions $\truth$ for the  
degree distribution, 
edgewise shared partner distribution,  
and geodesic distance distribution of networks of size 
$N \in \{25, 50, 75, 100\}$.}
\end{figure}

\begin{figure}[t]
\centering
\includegraphics[width = .75 \linewidth]{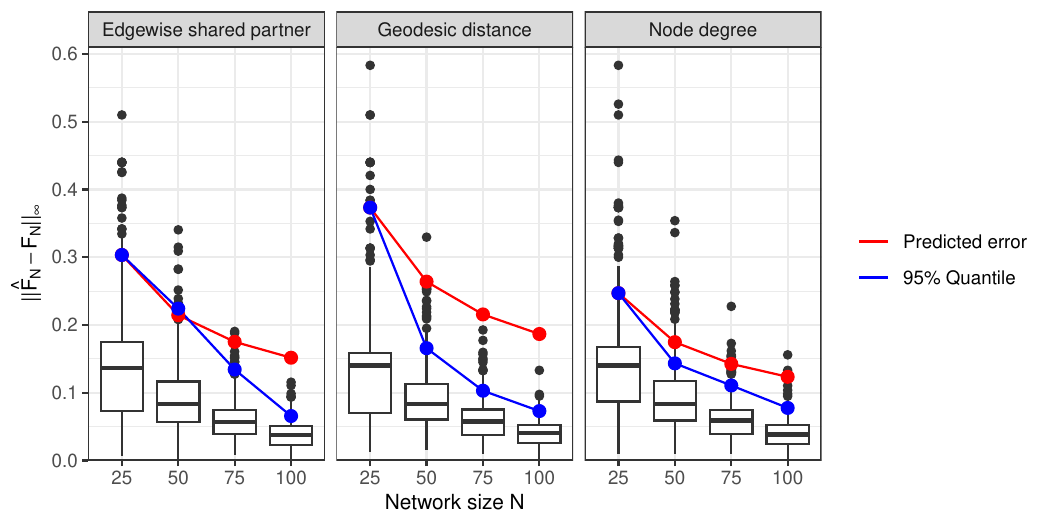}
\caption{\label{fig:sim1-2} Results of simulation study 1.
Boxplots summarizing the error $\norm{\emp - \truth}_{\infty}$ of 
of the degree distribution,  
edgewise shared partner distribution,  
and geodesic distance distribution 
of networks of size 
$N \in \{25, 50, 75, 100\}$ 
based on $500$ replications.
Rates of convergence for the $95\%$ quantile are predicted using Theorem \ref{thm:main1} and are indicated by the red line, 
compared with the actual $95\%$ quantile of the simulated errors.} 
\end{figure}

We focus on three different sequences of graph statistics,
including  the degree and edgewise shared partner distributions,
studied in Corollaries \ref{cor:1} and \ref{cor:2},
respectively,
as well as the geodesic distance distribution,
defined based on $(s_1(\bx), \ldots, s_{N-1}(\bx))$,  
where $s_k(\bx)$ is the number of pairs of nodes
with graph distance  equal to $k$.
All three 
are default diagnostic statistics 
in the goodness-of-fit method 
in the {\tt R} package {\tt ergm},
based on the work of \citet{Hunter2008}.

Simulation study 1 is conducted under the following conditions: 
\ben
\item The parameter vector is set to $(\theta_1, \theta_2, \theta_3) = (-3.5, .4, .75)$. 
\item The number of nodes varies from $N \in \{25, 50, 75, 100\}$. 
\item Networks $\bX$ are simulated using Markov-Chain Monte Carlo approaches \cite[see, e.g.,][]{Hunter2006}
in order to provide accurate approximations of $\truth$ and to generate network data sets for the simulation study.  
\item For each case of $N \in \{25, 50, 75, 100\}$, 
we generate $500$ replications and compute the empirical degree distribution $\emp$ and 
the error $\norm{\emp - \truth}_{\infty}$ using the approximation of $\truth$ outlined above. 
\een 
The results of Simulation study 1 are summarized in 
Figs. \ref{fig:sim1-1}
and \ref{fig:sim1-2}. 
We highlight two key elements of these results. 
First, 
the theoretical marginal distributions $\truth$ for each of the sequences of graph statistics 
considered in this simulation study are not constant in the network size $N$, 
as demonstrated by Fig. \ref{fig:sim1-1}. 
Each marginal distribution $\truth$ for $N \in \{25, 50, 75, 100\}$ 
was approximated using $2500$ sampled networks using the approach described above 
and to a maximum estimated standard error of under $.01$ 
and a total sum of estimated standard errors under $.01$. 
Second, 
the results of Fig. \ref{fig:sim1-2} 
demonstrate that the error 
$\norm{\emp - \truth}_{\infty}$
appropriately decays with high probability as a function of the network size $N$,
suggesting that the theoretical results of this work 
may be realized even in settings where networks are only modestly size,
noting that $N \leq 100$ in this study.

\subsection{Simulation study 2: sparse and dense $\beta$-models} 

The second simulation study we conduct focuses on the $\beta$-model 
\citep{Chatterjee2011}, 
which is related to the $p1$-model of \citet{Holland1981},
and posits a simple statistical model for degree heterogeneity in undirected random graphs. 
We will explore the $\beta$-model in the context of this work in both the dense and sparse graph regimes, 
where notably sparse variations of the $\beta$-model have garnered recent attention 
\citep[e.g,][]{Chen2021, Stewart2020}.
The $\beta$-model may be written down as follows:
\be
\label{beta1}
\mbP(\bX = \bx)
\= \dprod_{\{i,j\} \subset \mN} \, \mbP(X_{i,j} = x_{i,j}),
\ee
where 
\be
\label{beta2}
\log \dfrac{\mbP(X_{i,j} = 1)}{\mbP(X_{i,j} = 0)}
\= \theta_i + \theta_j, 
&& (\theta_i, \theta_j) \in \mbR^2, & \{i,j\} \subset \mN. 
\ee
The $\beta$-model has a straightforward interpretation. 
The log-odds of an edge in the network is equal to the sum $\theta_i + \theta_j$, 
where each node $i \in \mN$ is endowed with a parameters $\theta_i \in \mbR$. 
These parameters are interpreted as sociality parameters, 
where 
nodes $i \in \mN$ with larger $\theta_i \in \mbR$ will have a larger average node degree 
when compared with other nodes $j \in \mN$ for which $\theta_j < \theta_i$. 
As a result, 
the expected degrees of nodes under the $\beta$-model may exhibit significant heterogeneity. 

\begin{figure}[t]
\centering
\includegraphics[scale = .6]{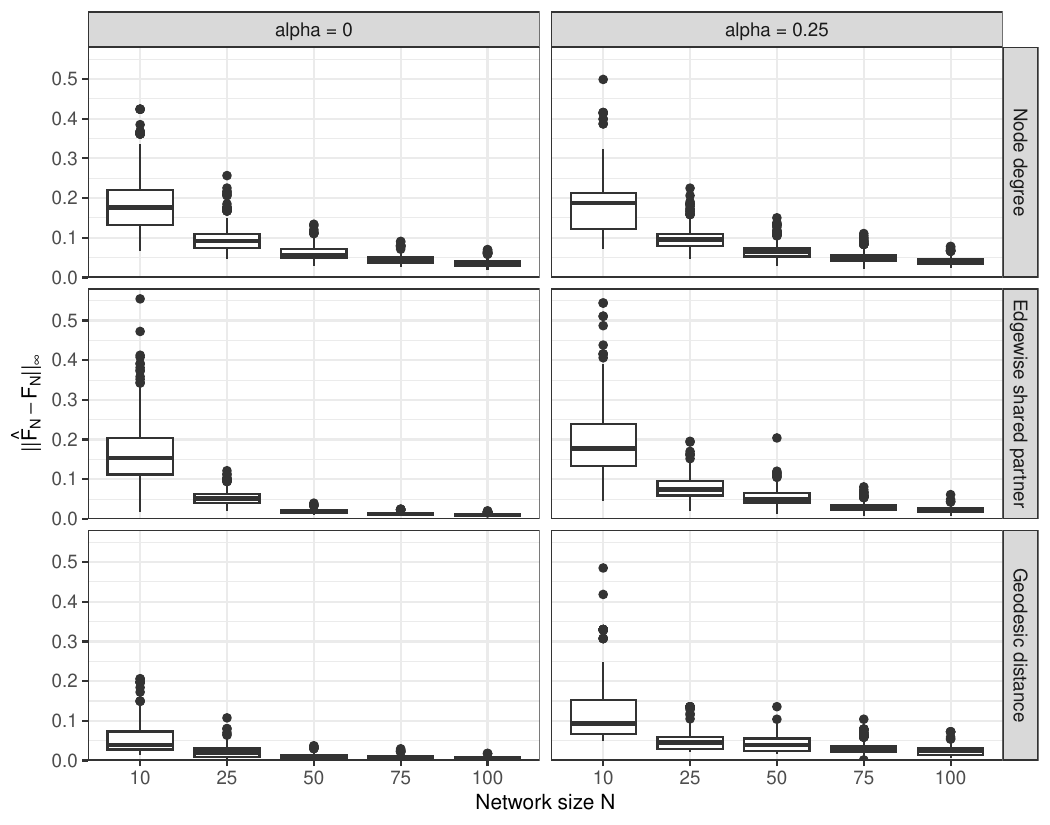} %
\caption{\label{fig:sim2} Results of simulation study 2.
Boxplots summarizing the error $\norm{\emp - \truth}_{\infty}$ of
of the degree distribution,
edgewise shared partner distribution,
and geodesic distance distribution
of networks of size
$N \in \{10, 25, 50, 75, 100\}$
based on $500$ replications.}
\end{figure}

Simulation study 2 is conducted under the following conditions:
\ben
\item The parameters in the vector $\btheta \in \mbR^N$ are independently simulated according to 
a normal distribution where $\theta_i \sim \text{N}(-\alpha \log N, \, 1)$
for all $i \in \mN$ and varying $\alpha \in \{0, .25\}$. 
\item The forms of \eqref{beta1} and \eqref{beta2} 
make it straightforward to simulate networks for approximating 
$\truth$ and simulating replicates in each simulation. 
\item For each case of $N \in \{10, 25, 50, 75, 100\}$ and $\alpha \in \{0, .25\}$,  
we generate $500$ replications and compute the empirical degree distribution $\emp$ and
the error $\norm{\emp - \truth}_{\infty}$ using the approximation of $\truth$.  
\een
Letting $\theta_i = \tilde\theta_i - \alpha \, \log(N)$ (for a fixed value of $\tilde\theta_i$ in a bounded set),
the value of $\alpha$ can be interpreted as follows: 
\beno
\log \, \dfrac{\mbP(X_{i,j} = 1)}{\mbP(X_{i,j} = 0)}
\= \tilde\theta_i + \tilde\theta_j - 2 \, \alpha \, \log(N), 
\ee
which implies that the expected degrees of nodes will grow slower than $O(N)$ when $\alpha > 0$ as  
\be
\label{eq:expected_degree}
\max\limits_{i \in \mN} \, \mbE \, \dsum_{j \in \mN \setminus \{i\}} \, X_{i,j} 
&\leq& C \; N^{1 - 2 \, \alpha},
&&\mbox{for some } C > 0. 
\ee
As a result, 
for $\alpha = 0$,
our simulation study is conducted in the dense graph regime where the expected number of edges in 
the network grows at a rate of $O(N^2)$,
whereas for $\alpha = .25$, 
the upper bound in \eqref{eq:expected_degree} demonstrates the expected degrees of nodes will 
satisfy $o(N)$,
implying the expected number of edges in the network will be $o(N^2)$.

The results of simulation study 2 are summarized in Fig. \ref{fig:sim2}. 
This simulation study highlights a key point that the theoretical results of this work 
are not affected by significant heterogeneity in node behavior, 
namely the node degrees. 
Theorems \ref{thm:main1} and \ref{thm:main2} reveal that the main factor which influences 
the error $\norm{\emp - \truth}_{\infty}$ is the dependence among the events of interest.  
Notably,
the $\beta$-model is within the scope of Corollary \ref{cor:bern}.
The results presented in Fig.  \ref{fig:sim2} show that 
the error $\norm{\emp - \truth}_{\infty}$ decays with the size of the network. 
The value of $\alpha$ appears to only influence the error associated with 
the empirical distributions of the edgewise shared partner and geodesic distance distributions,
given by the differences in the rates of decay between the first column and second column in Fig. \ref{fig:sim2},
with the case of $\alpha = 0$ presenting a faster decaying error $\norm{\emp - \truth}_{\infty}$ with the network size $N$.

\section{Application\label{sec:application}}

\begin{figure}[t]
\centering
\includegraphics[scale = .55]{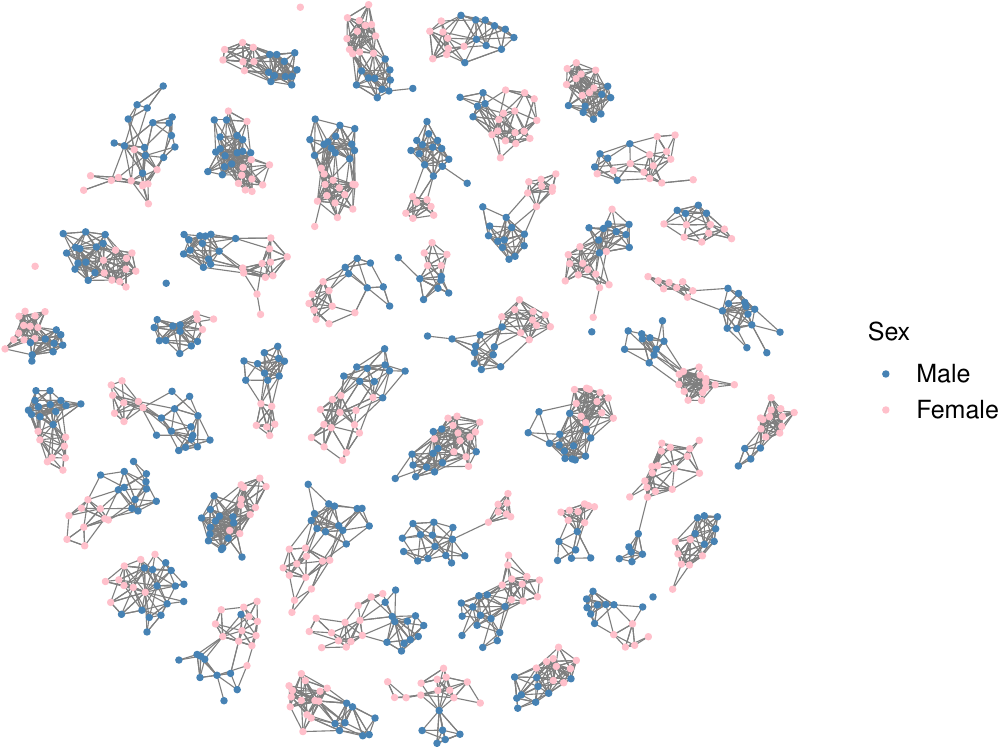} %
\caption{\label{fig:app_net}
A visualization of $44$ of the $304$ school classroom friendship networks in the school classes data set.}
\end{figure}

\begin{figure}[t]
\centering
\includegraphics[scale = .6]{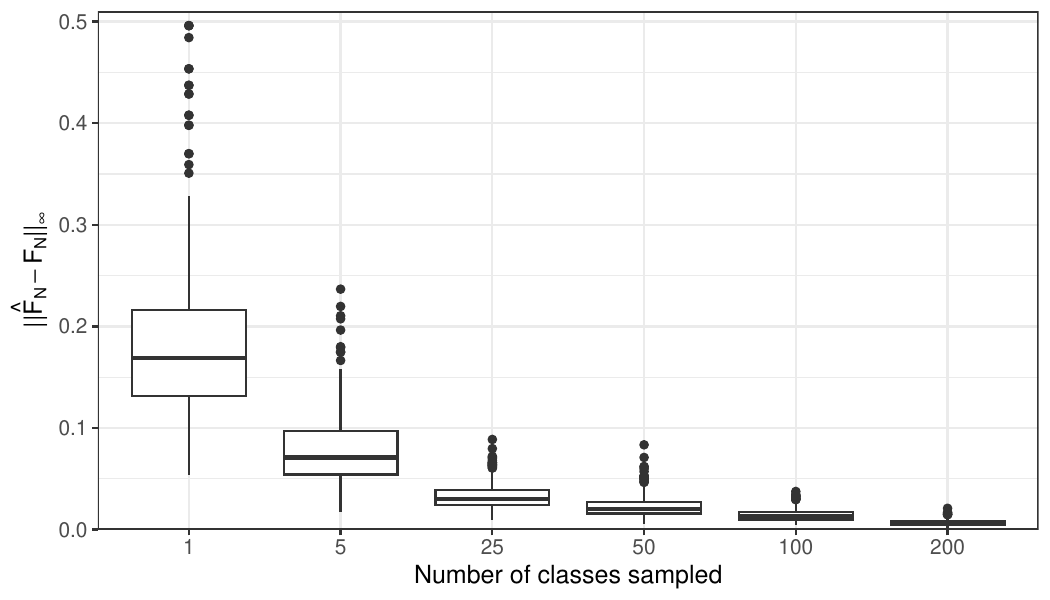}
\caption{\label{fig:app-1} 
Boxplots summarizing the error $\norm{\emp - \truth}_{\infty}$ of
of the empirical out-degree distribution 
based on random samples without replacement of size $K \in \{1, 5, 25, 50, 100, 200\}$ 
of the $304$ school classroom networks in the school classes data set.} 
\end{figure}

\begin{figure}[t]
\centering
\includegraphics[scale = .65]{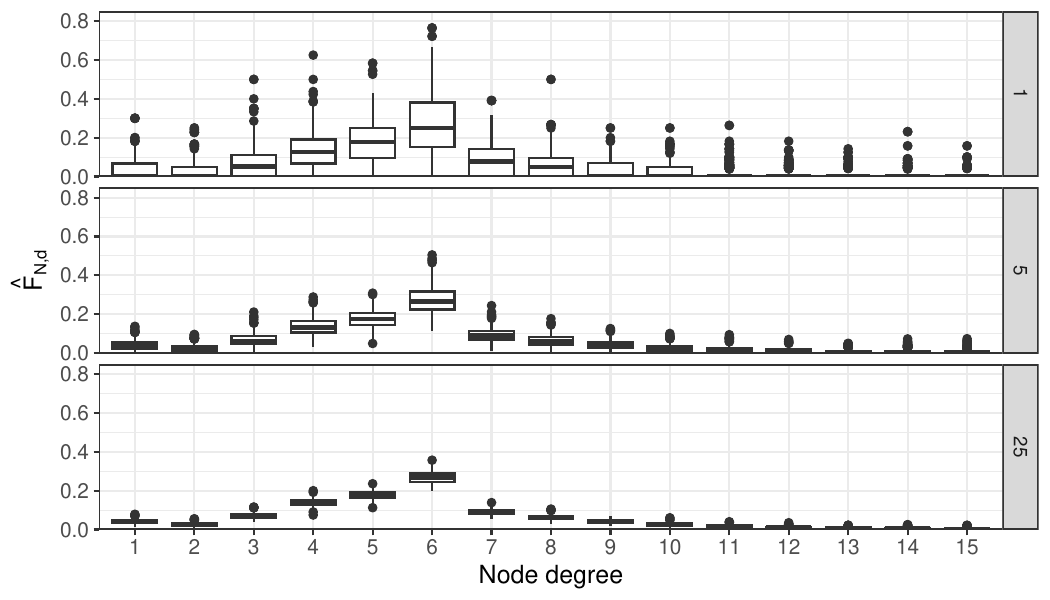}
\caption{\label{fig:app-2} 
Boxplots demonstrating the variability of the empirical out-degree distribution $\emp$ 
based on random samples without replacement of size $1$, $5$, and $25$ 
of the within-block subgraphs $\bX_{k,k}$ ($1 \leq k \leq 304$) 
based on  the $304$ school classroom networks in the school classes data set.} 
\end{figure}

We conclude with 
an application to a school classes network data set studied by  \citet{Stewart2019}.
The network data is a friendship network 
which consists of $N = 6594$ third grade students 
which are uniquely associated to one of $304$ classes 
within $176$ schools.
The network is directed as the data collection mechanism was a name generator 
which was ``Name people from your class that you would most like to play with,"
as well as information on the sex of students (recorded as male or female).  
A visualization of part of the network is given in Fig. \ref{fig:app_net}. 

The sampling mechanism for this data set includes only within-class edges,
which can be seen from the visualization of the network presented in Fig. \ref{fig:app_net}.
Due to the nature of the sampling design and the structure of the network, 
it is natural to assume a local dependence structure of the network 
\citep{Schweinberger2015,SchweinbergerStewart2020,stewart2024rates}, 
which was the approach taken in previous work studying this data set \citep{Stewart2019}. 
Local dependence random graphs can be thought of as generalizations of stochastic block models \citep{Holland1983}, 
where the set of nodes $\mN$ 
is partitioned into blocks or subpopulations $\mA_1, \ldots, \mA_K$ ($K \geq 2$). 
The term local dependence is due to the fact that joint distributions for local dependence random graphs are 
assumed to factorize with respect to the block-based subgraphs: 
\be
\label{eq:loc_dep}
\mbP(\bX = \bx)
\= \dprod_{1 \leq k \leq l \leq K} \, \mbP_{k,l}(\bX_{k,l} = \bx_{k,l}),
\ee
where $\bX_{k,l}$ is the vector of edge variables between nodes in block $\mA_k$ and $\mA_l$
and $\mbP_{k,l}$ is the marginal probability distribution of the subgraph $\bX_{k,l}$.  
As the form of \eqref{eq:loc_dep} suggests, 
edge variables within the same block-based subgraph $\bX_{k,l}$ ($1 \leq k \leq l \leq K$) 
are allowed to be dependent, 
but edges in different block-based subgraphs are assumed to be independent, 
hence the name of local dependence random graphs.
The relevance to this work lies in the fact that 
events of interest in \eqref{eq:stat} which are defined around the block-based subgraphs
will satisfy the main assumption of this work---the weak dependence
conditions outlined in Section \ref{sec:concentration} and discussed in Section \ref{sec:weak_dep}---will be satisfied,
due to the structure of a local dependence random graph.  
As an example, 
consider the within-block out-degree of nodes: 
\be
\label{eq:wc_deg}
\textrm{deg}^{+}_{i}(\bx)
&\coloneqq& \dsum_{j \in \mA_{z_i} \setminus \{i\}} \, x_{i,j},
&& i \in \mN,
\ee 
where $z_i \in \{1, \ldots, K\}$ denotes the community membership of node $i \in \mN$, 
i.e.,
$z_i = k$ implies $i \in \mA_k$. 
By the local dependence property, 
two within-block out-degree statistics $\textrm{deg}^{+}_{i}(\bX)$ 
and $\textrm{deg}^{+}_{j}(\bX)$ will be independent whenever $z_i \neq z_j$,
as each will be a function of non-overlapping subsets of edge variables in the network $\bX$, 
each of which are independent of the other.  
As a result,
it is straightforward to demonstrate in this case that 
$\dep \leq \max\{|\mA_1|, \ldots, |\mA_K|\}$. 
In other words, 
as long as the sizes of the blocks are bounded above or do not grow too quickly with the network size $N$, 
Theorems \ref{thm:main1} and \ref{thm:main2} will establish consistency for empirical distributions 
of sequences of graph statistics defined around block-based subgraphs,
notably including degree statistics 
edgewise shared partner statistics  
of within-block quantities,  
by an extension of Corollaries \ref{cor:1} and \ref{cor:2}. 
In this application, 
we would take $M_{\max} = \max\{|\mA_1|, \ldots, |\mA_K|\}$ and $\alpha_{\max} = 0$, 
in order to compare with the results of Corollaries \ref{cor:1} and \ref{cor:2}.

\hide{
The data we are studying includes only information on the within-block subgraphs 
$\bX_{k,k}$ ($k \in \{1, \ldots, K\}$).
It is useful here that the form of \eqref{eq:loc_dep} also establishes two useful properties for our purposes here:
\ben
\item First, 
by the independence of the block-based subgraphs $\bX_{k,l}$ ($1 \leq k \leq l \leq K$), 
the marginal distribution of the within-block subgraphs
$\bX_{k,k}$ ($k \in \{1, \ldots, K\}$) 
will have a convenient form: 
\beno
\mbP((\bX_{1,1}, \ldots, \bX_{K,K}) = (\bx_{1,1}, \ldots, \bx_{K,K}))
\= \dprod_{k=1}^{K} \, \mbP(\bX_{k,k} = \bx_{k,k}),
\ee
which allows us to be able to focus on the within-block subgraphs independently of the 
unobserved between-block subgraphs.  
\item Second, 
local dependence random graphs satisfy a weak form of projectivity  
based on around the block-based subgraphs \citep{SchweinbergerStewart2020}, 
which facilitates exploration of rates of convergence 
in a real-world data set as the collection of within-block subgraphs 
$\bX_{1,1}, \ldots, \bX_{K,K}$ are independent. 
In other words, 
we are able to subsample the block-based subgraphs by exploiting the independence, 
which facilitates an exploration of rates of convergence by utilizing this specific sampling mechanism.  
\een
}

Local dependence random graphs satisfy a weak form of projectivity
based on around the block-based subgraphs \citep{SchweinbergerStewart2020},
which facilitates exploration of rates of convergence
in a real-world data set as the collection of within-block subgraphs
$\bX_{1,1}, \ldots, \bX_{K,K}$ are independent.
This allows us to subsample the block-based subgraphs by exploiting the independence,
which facilitates an exploration of rates of convergence by utilizing this specific sampling mechanism.

One unique challenge that this data set holds is that there is a substantial amount of missing data. 
The response rate to the questionnaire was highly variable across the school classes, 
where the median response rate of students was $87\%$, 
with $44$ classes with responses rates of $100\%$ \citep{Stewart2019}.   
For every student which responded to the questionnaire, 
the out-edges of that student are observed,
which means that we observe the within-class out-degrees as defined in \eqref{eq:wc_deg}
for each respondent to the survey. 
As a result, 
we can define the empirical within-class out-degree distribution based on the subset of students responding to the questionnaire,
circumventing the issue of missing data. 

We visualize the error $\norm{\emp - \truth}_{\infty}$
based on subsampling without replacement individual class networks in Fig. \ref{fig:app-1},
with the number of classes being sampled in each iteration of subsampling ranging from $1$ class up to $200$ classes.
As this is a real-data application, 
the true distribution is unknown. 
However, 
for local dependence random graph models,
the quantity  
$\dep$
is bounded above by the size of the largest block size (here $33$), 
as discussed above. 
Hence, 
the results of Corollary \ref{cor:1} can establish the consistency of the empirical within-class out-degree distribution.
Since the number of nodes in this application is $N = 6594$, 
we treat the empirical out-degree distribution based on the entire network 
as an accurate approximation of $\truth$ for the purposes of this study
and explore changes in the error $\norm{\emp - \truth}_{\infty}$ 
when $\emp$ is based on a subsample of the classes in the entire network,
the results of which are visualized in Fig. \ref{fig:app-1}. 
Notably, 
by the time even just $25$ or $50$ classes are subsampled,
the empirical within-class out-degree distribution is relatively stable, 
showing low variability.
We further explore the variability of the within-class out-degree distributions of this network 
in Fig. \ref{fig:app-2},
which visualizes the variability of at each degree in the out-degree distribution across different amounts of subsampling. 
From these two plots, 
we can see that the empirical out-degree distribution becomes relatively stable past when 
$25$ school classes are subsampled.

\section{Conclusions\label{sec:5}}

This work has established the first statistical conditions which help to provide sufficient conditions 
under which one can expect empirical distributions of sequences of graph statistics to be uniformly consistent, 
by establishing non-asymptotic bounds on the error $\norm{\emp - \truth}_{\infty}$ which hold with high probability. 
Our results notably cover many statistics and charts used in network science applications 
which aim to study networks across a large variety of different domains and fields of study,  
emphasizing the importance of the development of statistical foundations that help us understand the 
properties of these statistics and charts. 
Moreover, 
we have demonstrated via mathematical applications that the probability threshold of our main results 
is sufficient to establish strong consistency of  empirical distributions,
in the sense that $\norm{\emp - \truth}_{\infty}$ converges almost surely to $0$ as the network size $N \to \infty$
in many applications of interest.
In particular, 
the theory we developed in this work covers  
a broad class of random graphs which allow edges to be dependent, 
making our results widely applicable to many applications.

The key to our approach lies in specifying weak dependence conditions which facilitate the derivation of 
exponential inequalities for the tails of distributions of sequences of graph statistics. 
We expect that the weak dependence assumptions would be satisfied in many applications.  
One case in point includes the local dependence random graphs \citep{Schweinberger2015,SchweinbergerStewart2020}, 
which was utilized in our application to the school classes network data set in Section \ref{sec:application}.  
In this class of models, 
it is possible to demonstrate that $\dep$ 
is bounded above provided the sizes of the local dependence neighborhoods are bounded above. 
An intuitive description of this assumption is provided in Section \ref{sec:weak_dep},
where we argued that, 
in the context of sociological applications, 
it is reasonable to assume that individuals are able to strongly influence only their local neighborhood,
giving rise to potentially strong local dependence,
but  weak global dependence. 
As demonstrated through Corollaries \ref{cor:1} and \ref{cor:2}, 
edge variables and events of interest defining the sequences of graph statistics 
may be strongly dependent on a local level (i.e., small subsets), 
but should not depend too strongly on a global level. 
To reiterate, 
we argue that such an assumption would be expected to be satisfied in many social networks 
(as well as domains beyond social network analysis), 
where individuals or relationships may have a strong local influence in the position of the network, 
but are unlikely to exert a strong global influence on the structure of the network.  

Lastly, 
our simulation studies and application verify that the general theory we have developed in this work 
may be realized in settings of networks which are modestly sized, 
as well as networks for which there may be strong triadic dependence  and 
significant node heterogeneity. 
Interestingly, 
our theoretical results are independent of certain complexities of models or probability distributions, 
essentially only requiring that a weak dependence assumption is satisfied. 
Combined, 
the empirical and theoretical results of this work help to provide a first comprehensive analysis 
of the conditions under which inferences drawn from empirical distributions of sequences of graph statistics 
may be expected to be consistent and informative
for networks with dependent edges.

\section*{Acknowledgments}

Jonathan R. Stewart was supported by 
National Science Foundation grant SES-2345043
and the Test Resource Management Center (TRMC) within the Office of the Secretary of Defense (OSD) under contract \#FA807518D0002,
and is grateful to two anonymous reviewers whose suggestions led to significant improvements in this work.

\section*{Appendix: Proofs of theoretical results}

\begin{proof}[\textbf{\upshape Proof of Lemma \ref*{lem:concentration}:}]
There are two concentration inequalities which must be proved. 

{\it First concentration inequality.}
The first concentration inequality is derived via a martingale decomposition.  
Following the definition of $\empnox$ given in \eqref{eq:emp},
we may write down the following for each $k \in \{0, 1, \ldots, p\}$:
\beno
\empk - \truthk
\= \dsum_{m=1}^{M} \, 
\left( \mbE\left[ \empk \,|\, \mF_{k,m}\right] - \mbE\left[\empk \,|\, \mF_{k,m-1}\right] \right), 
\ee
where $\mF_{k,m} \coloneqq \sigma\left(\one(\mG_{k,1}), \ldots, \one(\mG_{k,m})\right)$ ($m \in \{1, \ldots, M\}$)
is the filtration of the process based on the Bernoulli random variables defined in \eqref{eq:stat},
i.e.,
$\mF_{m,k}$ is the $\sigma$-field generated by 
the Bernoulli random variables 
$\one(\mG_{k,1}), \ldots, \one(\mG_{k,m})$ for each $k \in \{0, 1, \ldots, p\}$ and $m \in \{1, \ldots, M\}$. 
Applying the Azuma-Hoeffding inequality (e.g., Corollary 2.20, \citep{Wainwright2019}),
\be
\label{eq:a-h}
\mbP\left( \left| \empk - \truthk \right| \,\geq\, t \right)
&\leq& 2 \, \exp\left( - \dfrac{2 \, t^2}{\norm{\bXi_k}_2^2} \right),
&& t > 0, \;\; k \in \{0, 1, \ldots, p\},
\ee
defining $\bXi_k \coloneqq (\Xi_{k,1}, \ldots, \Xi_{k,M})$ 
($k \in \{0, 1, \ldots, p\}$)
with the definition,
for each $m \in \{1, \ldots, M\}$,  
\beno
\Xi_{k,m} 
&\coloneqq& \inf \, \left\{a \in [0, \infty) \,:\, 
\left|\mbE\left[ \empk \,|\, \mF_{k,m}\right] - \mbE\left[\empk \,|\, \mF_{k,m-1}\right] \right| \,\leq\, a 
\;\;\mbox{ holds }\;  \mbP\mbox{-a.s} \right\}. 
\ee
In the case where $\norm{\bXi_k}_2 = 0$,
we have the trivial bound of 
\beno
\mbP\left( \left| \empk - \truthk \right| \,\geq\, t \right) 
\= 0,
&& t > 0.
\ee
As a result, 
we proceed without loss under the assumption that $\norm{\bXi_k}_2 > 0$ ($k \in \{0, 1, \ldots, p\}$). 
By \eqref{eq:stat} and \eqref{eq:emp}, 
\be
\label{eq:mart_decomp}
\mbE\left[ \empk \,|\, \mF_{k,m}\right] - \mbE\left[\empk \,|\, \mF_{k,m-1}\right]
\= \dfrac{1}{M} \, \dsum_{i=m}^{M} \, \big( \mbE\left[ \one(\mG_{k,i}) \,|\,  \mF_{k,m}\right] 
- \mbE\left[ \one(\mG_{k,i}) \,|\,  \mF_{k,m-1}\right] \big),
\ee 
noting that $\mbE\left[ \one(\mG_{k,i}) \,|\,  \mF_{k,m}\right] 
= \one(\mG_{k,i})
= \mbE\left[ \one(\mG_{k,i}) \,|\,  \mF_{k,m-1}\right]$ 
($\mbP$-a.s.) for all $i < m$. 
We next bound 
\be
\label{eq:ind_bound}
\left|\mbE\left[ \one(\mG_{k,i}) \,|\,  \mF_{k,m}\right]
- \mbE\left[ \one(\mG_{k,i}) \,|\,  \mF_{k,m-1}\right] 
\right|
&\leq& \sup\limits_{(a,b) \in \{0,1\} \times \{0,1\}} \, 
\left| \mbE\left[ \one(\mG_{k,i}) \,|\,  \widetilde\mF_{k,m}^{(a)}\right]
- \mbE\left[ \one(\mG_{k,i}) \,|\,  \widetilde\mF_{k,m}^{(b)} \right] \right| \s\\
\= \left| \mbE\left[ \one(\mG_{k,i}) \,|\,  \widetilde\mF_{k,m}^{(0)}\right]
- \mbE\left[ \one(\mG_{k,i}) \,|\,  \widetilde\mF_{k,m}^{(1)} \right] \right|, 
\ee
defining 
$\widetilde\mF_{k,m}^{(\zeta)} \coloneqq \sigma\left(\one(\mG_{k,1}), \ldots, \one(\mG_{k,m-1}), \zeta\right)$ 
for $\zeta \in \{0, 1\}$, 
$m \in \{1, \ldots, M\}$, 
and $k \in \{0, 1, \ldots, p\}$. 
In words,
$\widetilde\mF_{k,m}^{(\zeta)}$ is the sub-$\sigma$-field of $\mF_{k,m}$ generated by the random variables 
$\one(\mG_{k,1}), \ldots, \one(\mG_{k,m})$ such that 
$\mbP\left(\one(\mG_{k,m}) = \zeta \,|\, \widetilde\mF_{k,m}^{(\zeta)}\right) = 1$.
Revisiting \eqref{eq:mart_decomp},
we obtain through the triangle inequality and \eqref{eq:ind_bound} 
\be
\label{eq:sum_bound}
\left| \dfrac{1}{M} \, \dsum_{i=m}^{M} \, \big( \mbE\left[ \one(\mG_{k,i}) \,|\,  \mF_{k,m}\right]
- \mbE\left[ \one(\mG_{k,i}) \,|\,  \mF_{k,m-1}\right] \big) \right|
\;\;\leq\;\; \dfrac{1}{M} \,  \dsum_{i=m}^{M} \left| \mbE\left[ \one(\mG_{k,i}) \,|\,  \mF_{k,m}\right]
- \mbE\left[ \one(\mG_{k,i}) \,|\,  \mF_{k,m-1}\right] \right| \s\\
\quad\quad \leq\;\;  
\dfrac{1}{M} \, \dsum_{i=m}^{M} \, \left| \mbE\left[ \one(\mG_{k,i}) \,|\,  \widetilde\mF_{k,m}^{(0)}\right]
- \mbE\left[ \one(\mG_{k,i}) \,|\,  \widetilde\mF_{k,m}^{(1)} \right] \right| 
\;\;=\;\; \dfrac{1}{M} \, \dsum_{i=m}^{M} \, \left| \mbP\left(\mG_{k,i} \,|\,  \widetilde\mF_{k,m}^{(0)}\right)
- \mbP\left(\mG_{k,i} \,|\,  \widetilde\mF_{k,m}^{(1)} \right) \right| \s\\ 
\quad\quad \leq\;\; \dfrac{1}{M} \, \left( 1 + \dsum_{i=m+1}^{M} \vartheta_{m,i,k} \right),
\ee
defining,
for each $k \in \{0, 1, \ldots, p\}$ and $m \in \{1, \ldots, M\}$, 
\be
\label{eq:vthetas}
\vartheta_{m,i,k}
&\coloneqq& \dfrac{1}{2} \, \dsum_{v=0}^{1} \, \left| 
\mbP\left(\one(\mG_{k,i}) = v \,|\, \widetilde\mF_{k,m}^{(0)} \right)
- \mbP\left(\one(\mG_{k,i}) = v \,|\, \widetilde\mF_{k,m}^{(1)} \right) \right|,
&& i \in \{1, \ldots, M\}. 
\ee
In words, 
$\vartheta_{m,i,k}$ is the total variation distance between the conditional probability distributions 
of the Bernoulli random variable $\one(\mG_{k,i})$
under the conditioning sub-$\sigma$-fields $\widetilde\mF_{k,m}^{(0)}$ 
and $\widetilde\mF_{k,m}^{(1)}$. 
Observe that $\vartheta_{m,i,k} = 1$ when we have $i = m$,
and $\vartheta_{m,i,k} = 0$ for all $i < m$.  
Note that each $\vartheta_{m,i,k}$ in \eqref{eq:vthetas} is a random variable,
because the total variation distances involve conditional expectations which are random variables. 
In order to obtain a bound on the 
$\vartheta_{m,i,k}$ 
which holds with probability $1$,
we define,
for ease of presentation, 
the Bernoulli random variables 
$B_{k,i} \coloneqq  \one(\mG_{k,i})$,
$k \in \{0, 1, \ldots, p\}$, $i \in \{1, \ldots, M\}$,
and the conditional probability distribution of $B_{k,j}$
conditional on
$(B_{k,i}, \ldots, B_{k,i})$,
$j \in \{1, \ldots, M\}$, 
by
\beno
\mbP_{i,j,k}^{\bb}(v)
&\coloneqq& \mbP\left(B_{k,j} = v \;|\; (B_{k,1}, \ldots, B_{k,i}) = \bb\right),
& v \in \{0,1\}, \; \bb \in \{0,1\}^{i},
\;\;  (i,k) \in \{1, \ldots, M\} \times \{0, 1, \ldots, p\},
\ee
and define
\beno
\mdelta &\coloneqq&
\max\limits_{(\bb,\bb^\prime) \in \{0, 1\}^{i} \times \{0,1\}^{i} \,:\, b_t = b_t^\prime, \, t < i} \;
d_{\tv}\left(\mbP_{i,j,k}^{\bb}, \, \mbP_{i,j,k}^{\bb^\prime}\right)
&&  i \in \{1, \ldots, M\}, \; k \in \{0, 1, \ldots p\}. 
\ee
We can then bound each $\vartheta_{m,i,k}$ in \eqref{eq:vthetas} by 
\be
\label{eq:tv_bound}
\vartheta_{m,i,k}
&\leq& 
\max\limits_{(\bb,\bb^\prime) \in \{0, 1\}^{i} \times \{0,1\}^{i} \,:\, b_t = b_t^\prime, \, t < i} \;
d_{\tv}\left(\mbP_{m,i,k}^{\bb}, \, \mbP_{m,i,k}^{\bb^\prime}\right)
&\eqqcolon& \delta_{m,i,k}.  
\ee
From \eqref{eq:sum_bound}, 
this results in the bound 
\be
\label{eq:Xi_bound}
\Xi_{k,m}
&\leq& \dfrac{1}{M} \, \left(1 + \dsum_{i = m +1}^{M} 
\delta_{m,i,k} \right), 
&& k \in \{0, 1, \ldots, p\}, \;\; m \in \{1, \ldots, M\}. 
\ee
Using \eqref{eq:tv_bound} and \eqref{eq:Xi_bound},
we revisit \eqref{eq:a-h} to obtain,
for $t > 0$, 
the inequality 
\beno
\mbP\left( \left| \empk - \truthk \right| \,\geq\, t \right)
&\leq& 2 \, \exp\left( - \dfrac{2 M  t^2}{\mcD_{k,N}} \right),
&& k \in \{0, 1, \ldots, p\},
\ee 
defining,
for all $N \in \{3, 4, \ldots\}$ 
and all $k \in \{0, 1, \ldots, p\}$, 
\beno
\mcD_{N,k}
&\coloneqq& 
\dfrac{1}{M} \,  
\dsum_{m=1}^{M} \, \left( 1 + \dsum_{i = m+1}^{M} \, 
\delta_{m,i,k}
\right)^2,  
\ee
observing that in general $M$ will depend on $N$. 
By a union bound over the $p+1$ 
components of $\emp - \truth$,
we obtain 
\beno
\mbP\left( \norm{\emp - \truth}_{\infty} \,\geq\, t \right)
&\leq& 2 \, \exp\left( - \dfrac{2 M  t^2}{\dep} + \log(1+p) \right),
&& t > 0, 
\ee
defining $\dep \coloneqq \max\{\mcD_{N,0}, \mcD_{N,1}, \ldots, \mcD_{N,p}\}$.

\s 

{\it Second concentration inequality.} 
We start with a union bound over the $p+1$
components of $\emp - \truth$: 
\beno
\mbP\left(\norm{\emp - \truth}_{\infty} \,\geq\, t \right)
&\leq& \dsum_{k=0}^{p} \mbP\left( |\empk - \truthk| \,\geq\, t \right),
&& t > 0. 
\ee
Next,
for each $k \in \{0, 1, \ldots, p\}$,
we apply Chebyshev's inequality to obtain 
\be
\label{thm1_start}
\mbP\left( |\empk - \truthk| \,\geq\, t \right)
\= \mbP\left( \left| \dfrac{1}{M} \, \dsum_{i=1}^{M} \, B_{k,i} 
- \dfrac{1}{M} \, \dsum_{i=1}^{M} \, \mbE \, B_{k,i} \right| \,\geq\, t\right) \s\\
&\leq& \dfrac{1}{M^2 \, t^2} \, \left( \dsum_{i=1}^{M} \, \left[ 
\var \, B_{k,i} + \dsum_{j \in \{1, \ldots, M\} \setminus \{i\}} \, \cov(B_{k,i}, \, B_{k,j}) \right] \right) \s\\
&\leq& \dfrac{1}{M^2 \, t^2} \, \dsum_{i=1}^{M} \, \mbP(B_{k,i} = 1)
+ \dfrac{1}{M^2 \, t^2} \, \dsum_{i=1}^{M} \, \dsum_{j \in \{1, \ldots, M\} \setminus \{i\}} \, \cov(B_{k,i}, \, B_{k,j}) \s\\
\= \dfrac{\muk}{M \, t^2} + \dfrac{1}{M^2 \, t^2} \, \dsum_{i=1}^{M} \, \dsum_{j \in \{1, \ldots, M\} \setminus \{i\}} \, \cov(B_{k,i}, \, B_{k,j}), 
\ee 
with the definition 
$\muk \coloneqq (1/M)  \sum_{i=1}^{M} \mbP(B_{k,i} = 1)$ ($k \in \{0, 1, \ldots, p\}$). 
The second inequality presents two different upper bounds,
which we derive separately. 
For the first bound, 
we have
\beno
\mbP\left(\norm{\emp - \truth}_{\infty} \,\geq\, t \right)
&\leq& \dsum_{k=0}^{p} \, \left(
\dfrac{\muk}{M \, t^2} + \dfrac{1}{M^2 \, t^2} \, \dsum_{i=1}^{M} \, \dsum_{j \in \{1, \ldots, M\} \setminus \{i\}} \, \cov(B_{k,i}, \, B_{k,j}) \right)
&\leq& \dfrac{1 + \mcC_N}{M \, t^2}, 
\ee 
defining 
\beno
\mcC_N
&\coloneqq& \dfrac{1}{M} \;
\dsum_{i=1}^{M} \,
\dsum_{j \in \{1, \ldots, M\} \setminus \{i\}} \,
\dsum_{k=0}^{p} \,
\cov(B_{k,i}, \, B_{k,j}),
&& k \in \{0, 1, \ldots, p\},
\ee
and noting that $\sum_{k=0}^{p} \, \mu_k = 1$,
because 
\beno
\dsum_{k=0}^{p} \, \muk
\= \dsum_{k=0}^{p} 
\left( \dfrac{1}{M} \, \dsum_{i=1}^{M} \, \mbP(B_{k,i} = 1) \right)
\=  
\dfrac{1}{M} \, \dsum_{i=1}^{M} \, \dsum_{k=0}^{p} \, \mbP(B_{k,i} = 1)
\= \dfrac{1}{M} \, \dsum_{i=1}^{M} \, 1 
\= 1. 
\ee 
Hence,
the first upper bound is verified:  
\be
\label{eq:first_bound}
\mbP\left(\norm{\emp - \truth}_{\infty} \,\geq\, t \right)
&\leq& \dfrac{1 + \mcC_N}{M \, t^2},
&& t > 0. 
\ee
For the second upper bound, 
we start from \eqref{thm1_start},
by writing
\beno
\cov(B_{k,i}, \, B_{k,j})
\;\;=\;\;\mbE[B_{k,i} \, B_{k,j}] - \mbE[B_{k,i}] \, \mbE[B_{k,j}]
\;\;=\;\; \mbP(\{B_{k,i} = 1\} \cap \{B_{k,j} = 1\}) - \mbP(B_{k,i} = 1) \, \mbP(B_{k,j} = 1) \s\\ 
\;\;=\;\; \mbP(B_{k,i} = 1) \, \mbP(B_{k,j} = 1 \,|\, B_{k,i} = 1)  - \mbP(B_{k,i} = 1) \, \mbP(B_{k,j} = 1)
\;\;=\;\; \mbP(B_{k,i} = 1) \left( \mbP(B_{k,j} = 1 \,|\, B_{k,i} = 1) - \mbP(B_{k,j} = 1) \right). 
\ee
This yields,
from \eqref{thm1_start},
the summation 
\beno
\dfrac{1}{M^2 \, t^2} \, \dsum_{i=1}^{M} \, \dsum_{j \in \{1, \ldots, M\} \setminus \{i\}} \, \cov(B_{k,i}, \, B_{k,j})
\= \dfrac{1}{M^2 \, t^2} \,
\dsum_{i=1}^{M} \, \mbP(B_{k,i} = 1) \, 
\dsum_{j \in \{1, \ldots, M\} \setminus \{i\}}
\left( 
\mbP(B_{k,j} = 1 \,|\, B_{k,i} = 1) - \mbP(B_{k,j} = 1)
\right),
\ee
which implies 
\beno
\mbP\left(\norm{\emp - \truth}_{\infty} \,\geq\, t \right)
&\leq& \dsum_{k=0}^{p} \, \left(
\dfrac{\muk}{M \, t^2} + \dfrac{1}{M^2 \, t^2} 
\dsum_{i=1}^{M} \, \mbP(B_{k,i} = 1)
 \dsum_{j \in \{1, \ldots, M\} \setminus \{i\}} 
\left( 
\mbP(B_{k,j} = 1 \,|\, B_{k,i} = 1) - \mbP(B_{k,j} = 1)
\right)
\right) \s\\
&\leq& \dfrac{1}{M \, t^2} + 
\dfrac{1}{M^2 \, t^2}
\dsum_{i=1}^{M} \, 
\dsum_{k=0}^{p} \, 
\mbP(B_{k,i} = 1) \, 
\dsum_{j \in \{1, \ldots, M\} \setminus \{i\}} \, 
\left( 
\mbP(B_{k,j} = 1 \,|\, B_{k,i} = 1) - \mbP(B_{k,j} = 1)
\right),
\ee
recalling that  $\sum_{k=0}^{p} \mu_k = 1$.
\hide{ 
For $i \in \{1, \ldots, M\}$, 
the Cauchy-Schwarz inequality and the inequality $\norm{\bv}_2 \leq \norm{\bv}_1$
yields 
\beno 
&& 
\dsum_{k=0}^{p} \,
 \mbP(B_{k,i} = 1)
\dsum_{j \in \{1, \ldots, M\} \setminus \{i\}} \,
\mbE \left[ \, \left( \mbP\left(B_{k,j} = 1 \,|\, \sigma(B_{k,i})\right) - \mbP\left(B_{k,j} = 1\right) \right)
\right] \s\\ 
&\leq& \left( 
 \dsum_{k=0}^{p}  \mbP(B_{k,i} = 1) \right) 
\s
\left( \dsum_{k=0}^{p} \,
\dsum_{j \in \{1, \ldots, M\} \setminus \{i\}} \, 
\mbE \left[ \, \left( \mbP\left(B_{k,j} = 1 \,|\, \sigma(B_{k,i})\right) - \mbP\left(B_{k,j} = 1\right) \right)
\right] \right) \s\\
\= \mbE\left[ \, \dsum_{j \in \{1, \ldots, M\} \setminus \{i\}} \,
\dsum_{k=0}^{p} \, 
\left( \mbP\left(B_{k,j} = 1 \,|\, \sigma(B_{k,i})\right) - \mbP\left(B_{k,j} = 1\right) \right) \right] \s\\
&\leq& \max\limits_{\bb_i \in \{0,1\}^{p+1}} \,
\dsum_{j \in \{1, \ldots, M\} \setminus \{i\}} \,
d_{\tv}\left( \pi_{j}, \; \pi_{j\,|\,i}^{\bb_i}\right),  
\ee
recalling the definitions  
$\bB_i \coloneqq (B_{0,i}, B_{1,i}, \ldots, B_{p,i})$ ($i \in \{1, \ldots, M\}$),
\beno
\pi_{j\,|\,i}^{\bb_i}(\bb_j)
&\coloneqq& \mbP(\bB_j = \bb_j \,|\, \bB_i = \bb_i),
&& j \in \{1, \ldots, M\} \setminus \{i\}, \; i \in \{1, \ldots, M\}, \s \\
\pi_{j}(\bb_j)
&\coloneqq& \mbP(\bB_j = \bb_j),
&& j \in \{1, \ldots, M\}. 
\ee
}
As a result,
defining 
\beno
\Delta_{N} &\coloneqq& 
\dfrac{1}{M} \,
\dsum_{i=1}^{M} \,
\dsum_{k=0}^{p} \,
\mbP(B_{k,i} = 1) \,
\dsum_{j \in \{1, \ldots, M\} \setminus \{i\}} \,
\left(
\mbP(B_{k,j} = 1 \,|\, B_{k,i} = 1) - \mbP(B_{k,j} = 1)
\right),
\ee 
we prove the second upper bound 
\be
\label{eq:second_bound}
\mbP\left(\norm{\emp - \truth}_{\infty} \,\geq\, t \right)
&\leq& \dfrac{1 + \Delta_{N}}{M \, t^2}.  
\ee
Combining both \eqref{eq:first_bound} and \eqref{eq:second_bound},
we prove the second inequality 
\beno
\mbP\left( \norm{\emp - \truth}_{\infty} \,\geq\, t \right)
&\leq& \dfrac{1 + \min\left\{\mcC_N, \; \Delta_N\right\}}{M \, t^2}, 
&& t > 0. 
\ee
\end{proof}

\s

\begin{proof}[\textbf{\upshape Proof of Proposition \ref*{prop:1}:}]
By Definition \ref{def2},
\beno
\Delta_{N} &\coloneqq&
\dfrac{1}{M} \,
\dsum_{i=1}^{M} \,
\dsum_{k=0}^{p} \,
\mbP(B_{k,i} = 1) \,
\dsum_{j \in \{1, \ldots, M\} \setminus \{i\}} \,
\left(
\mbP(B_{k,j} = 1 \,|\, B_{k,i} = 1) - \mbP(B_{k,j} = 1)
\right).
\ee
Define,
for each $i \in \{1, \ldots, M\}$, 
the random vector  
$\bB_i \coloneqq (B_{0,i}, B_{1,i}, \ldots, B_{p,i}) \in \{\bb \in \{0,1\}^{p+1} : \norm{\bb}_1 = 1\}$,
and 
\beno
\pi_{j\,|\,i}^{\bb_i}(\bb_j)
&\coloneqq& \mbP(\bB_j = \bb_j \,|\, \bB_i = \bb_i),
&& j \in \{1, \ldots, M\} \setminus \{i\}, \; i \in \{1, \ldots, M\}, \s \\
\pi_{j}(\bb_j)
&\coloneqq& \mbP(\bB_j = \bb_j),
&& j \in \{1, \ldots, M\}.
\ee
Since the events $\mG_{0,i},\ldots, \mG_{p,i}$ are mutually exclusive, 
$\norm{\bB_i}_1 = 1$ holds with probability $1$.  
By the triangle inequality,  
\beno
\Delta_{N}
\;\;\leq\;\; \dfrac{1}{M} \,
\dsum_{i=1}^{M} \,
\dsum_{j \in \{1, \ldots, M\} \setminus \{i\}} \,
\dsum_{k=0}^{p} \,
\mbP(B_{k,i} = 1) \,
\left| \mbP(B_{k,j} = 1 \,|\, B_{k,i} = 1) - \mbP(B_{k,j} = 1) \right| \s\\ 
\;\;\leq\;\; \dfrac{1}{M} \,
\dsum_{i=1}^{M} \,
\dsum_{j \in \{1, \ldots, M\} \setminus \{i\}} \,
\dsum_{\bb_i \in \{\bb \in \{0,1\}^{p+1} \,:\, \norm{\bb}_1 = 1\}} \,
\mbP(\bB_{i} = \bb_i) \; d_{\tv}\left(\pi_{j\,|\,i}^{\bb_i}, \, \pi_{j} \right)
\;\;=\;\; \dfrac{1}{M} \,
\dsum_{i=1}^{M} \,
\dsum_{j \in \{1, \ldots, M\} \setminus \{i\}} \, \mbE \, d_{\tv}\left(\pi_{j\,|\,i}^{\bB_i}, \, \pi_{j} \right).  
\ee 
\end{proof}

\s

\begin{proof}[\textbf{\upshape Proof of Theorem \ref*{thm:main1}:}] 
The assumptions of Theorem \ref{thm:main1} ensure the assumptions of Lemma \ref{lem:concentration} are satisfied.
To obtain the first result, 
we apply Lemma \ref{lem:concentration},
for all $\epsilon > 0$,
to obtain the inequality 
\beno
\mbP\left( \norm{\emp - \truth}_{\infty} \,\geq\, \epsilon \right)
&\leq& 2 \, \exp\left( - \dfrac{2 \, M \, \epsilon^2}{\dep} + \log(1+p) \right),
\ee
which in turn implies through the complement rule that 
\beno
\mbP\left( \norm{\emp - \truth}_{\infty} \,<\, \epsilon \right)
&\geq& 1 - 2 \, \exp\left( - \dfrac{2 \, M \, \epsilon^2}{\dep} + \log(1+p) \right) 
&\geq& 1 - 2 \, \exp\left( - \dfrac{2 \, M \, \epsilon^2}{\dep} + \log(\max\{M, \, 1+p\}) \right). 
\ee
Choosing 
\beno
\epsilon
\= \sqrt{\dfrac{3}{2}} \,
\sqrt{\dfrac{\dep \, \log(\max\{M, \, 1+p\})}{M}}
\ee
establishes that 
\beno
\mbP\left( \norm{\emp - \truth}_{\infty} \,<\, \sqrt{\dfrac{3}{2}} \,
\sqrt{\dfrac{\dep \, \log(\max\{M, \, 1+p\})}{M}} \right)
&\geq& 1  - \dfrac{2}{\max\{M, \, 1+p\}^2}. 
\ee
To obtain the second result, 
we apply Lemma \ref{lem:concentration},
for all $\epsilon > 0$,
to obtain the inequality
\beno
\mbP\left( \norm{\emp - \truth}_{\infty} \,\geq\, \epsilon \right)
&\leq& \dfrac{1 +\min\left\{\mcC_N, \; \Delta_N\right\}}{M \, \epsilon^2}.
\ee
Let $\alpha \in (0, 1)$. 
Choosing 
\beno
\epsilon 
\= \sqrt{\dfrac{1 + \min\left\{|\mcC_N|, \; |\Delta_N|\right\}}{\alpha \, M}}
&\in& (0, \infty) 
\ee
establishes through the complement rule that 
\beno
\mbP\left(  \norm{\emp - \truth}_{\infty} \,<\, \sqrt{\dfrac{1 + \min\left\{|\mcC_N|, \; |\Delta_N|\right\}}{\alpha \, M}} \right)
&\geq& 1 - \alpha,
\ee
for all significance levels $\alpha \in (0, 1)$. 
\end{proof}

\s

\begin{proof}[\textbf{\upshape Proof of Theorem \ref*{thm:main2}:}]
By the complement rule, 
\be
\label{eq:last_step}
\mbP\left( \norm{\emp - \truth}_{\infty} \,<\, \epsilon \right)
\= 1 - \mbP\left( \norm{\emp - \truth}_{\infty} \,\geq\, \epsilon \right).
\ee
We lower bound the probability 
$\mbP\left( \norm{\emp - \truth}_{\infty} \,<\, \epsilon \right)$ 
by upper bounding the probability 
$\mbP\left( \norm{\emp - \truth}_{\infty} \,\geq\, \epsilon \right)$. 
Applying the law of total probability, 
we obtain, 
for all $\epsilon > 0$,  
\be
\label{eq:pbound1}
\mbP\left( \norm{\emp - \truth}_{\infty} \,\geq\, \epsilon \right) 
\;\;=\;\; \mbP\left( \left\{ \norm{\emp - \truth}_{\infty} \,\geq\, \epsilon \right\}
\,\cap\, \left\{\bX \in \mbX_0\right\} 
\right) 
+  \mbP\left( \left\{ \norm{\emp - \truth}_{\infty} \,\geq\, \epsilon \right\}
\,\cap\, \left\{\bX \in \mbX_0^c\right\}
\right) \s\\
\;\;\leq\;\; \mbP\left(\norm{\emp - \truth}_{\infty} \,\geq\, \epsilon \,|\, \bX \in \mbX_0\right)
+ \mbP(\bX \in \mbX_0^c) 
\;\;\leq\;\; \mbP\left(\norm{\emp - \truth}_{\infty} \,\geq\, \epsilon \,|\, \bX \in \mbX_0\right) + r(N),
\ee
using the bound in \eqref{eq:A2-1} which implies $\mbP(\bX \in \mbX_0^c) \leq r(N)$.
Note that the assumption of \eqref{eq:A2-1} ensures that the conditional probability given in 
\eqref{eq:pbound1} is well-defined by assuming that $r(N) \in (0, 1)$ so that 
$\mbP(\bX \in \mbX_0) \geq 1 - r(N) > 0$. 
We upper bound the conditional probability 
$\mbP\left(\norm{\emp - \truth}_{\infty} \,\geq\, \epsilon \,|\, \bX \in \mbX_0\right)$ 
by manipulating  the event of interest:  
\beno
\norm{\emp - \truth}_{\infty} 
\= \norm{ \, \emp - \mbE[\emp \,|\, \bX \in \mbX_0] \, \mbP(\bX \in \mbX_0) 
- \mbE[\emp \,|\, \bX \in \mbX_0^c] \, \mbP(\bX \in \mbX_0^c) \, }_{\infty} \s\\
&\leq& \norm{\, \emp - \mbE[\emp \,|\, \bX \in \mbX_0] \, \mbP(\bX \in \mbX_0) \, }_{\infty}
+ \norm{\, \mbE[\emp \,|\, \bX \in \mbX_0^c] \, \mbP(\bX \in \mbX_0^c) \, }_{\infty} \s\\
&\leq& \norm{ \, \emp - \mbE[\emp \,|\, \bX \in \mbX_0] \, \mbP(\bX \in \mbX_0) \, }_{\infty} 
+ r(N), 
\ee
where we apply the law of total expectation in the first line,
obtain the inequality in the second line from the triangle inequality,  
and obtain the last inequality 
from the fact that,
for all $k \in \{0, 1, \ldots, p\}$, 
\beno 
\left|\mbE[\empk \,|\, \bX \in \mbX_0^c] \, \mbP(\bX \in \mbX_0^c) \right|  
&\leq&  \mbP(\bX \in \mbX_0^c)
&\leq& r(N),
\ee
since $\mbE[\empk \,|\, \bX \in \mbX_0^c] \in [0, 1]$ and using \eqref{eq:A2-1}.  
Next, 
we apply the triangle inequality to obtain 
\beno
\norm{ \, \emp - \mbE[\emp \,|\, \bX \in \mbX_0] \, \mbP(\bX \in \mbX_0) \, }_{\infty} \s\\
\quad\quad\leq\;\; \norm{ \, \emp - \mbE[\emp \,|\, \bX \in \mbX_0] \, }_{\infty}
+ \norm{ \, \mbE[\emp \,|\, \bX \in \mbX_0] - \mbE[\emp \,|\, \bX \in \mbX_0] \, \mbP(\bX \in \mbX_0) \, }_{\infty} \s\\
\quad\quad=\;\; \norm{ \, \emp - \mbE[\emp \,|\, \bX \in \mbX_0] \, }_{\infty} 
+ \norm{ \, \mbE[\emp \,|\, \bX \in \mbX_0] \, }_{\infty} \, (1 - \mbP(\bX \in \mbX_0)) \s\\
\quad\quad\leq\;\; \norm{ \, \emp - \mbE[\emp \,|\, \bX \in \mbX_0] \, }_{\infty} + \mbP(\bX \in \mbX_0^c)
\;\;\leq\;\; \norm{ \, \emp - \mbE[\emp \,|\, \bX \in \mbX_0] \, }_{\infty} + r(N),  
\ee
again using the fact that $\mbE[\empk \,|\, \bX \in \mbX_0] \in [0, 1]$ and using \eqref{eq:A2-1}.
Altogether, 
we have shown the inequality 
\beno
\mbP\left(\norm{\emp - \truth}_{\infty} \,\geq\, \epsilon \,|\, \bX \in \mbX_0\right)
&\leq& \mbP\left(\norm{ \, \emp - \mbE[\emp \,|\, \bX \in \mbX_0] \, }_{\infty} + 2 \, r(N) 
\,\geq\, \epsilon \,|\, \bX \in \mbX_0\right). 
\ee 
Next,
applying Lemma \ref{lem:AB_bound}, 
we have 
\beno
\mbP\left(\norm{ \, \emp - \mbE[\emp \,|\, \bX \in \mbX_0] \, }_{\infty} + 2 \,r(N) \,\geq\, \epsilon \,|\, \bX \in \mbX_0\right)
\s\\
\quad\quad\leq\;\; \mbP\left(\norm{ \, \emp - \mbE[\emp \,|\, \bX \in \mbX_0] \, }_{\infty} \,\geq\, \dfrac{\epsilon}{3} 
\,\Big|\, \bX \in \mbX_0\right)
+ 2 \, \mbP\left(r(N) \,\geq\, \dfrac{\epsilon}{3} \,\Big|\, \bX \in \mbX_0\right) \s\\
\quad\quad=\;\; \mbP\left(\norm{ \, \emp - \mbE[\emp \,|\, \bX \in \mbX_0] \, }_{\infty} \,\geq\, \dfrac{\epsilon}{3}
\,\Big|\, \bX \in \mbX_0\right) + 
2 \, \one\left( r(N) \,\geq\, \dfrac{\epsilon}{3} \right).  
\ee 
For the first term, 
we apply Lemma \ref{lem:concentration} to obtain 
\beno
\mbP\left(\norm{ \, \emp - \mbE[\emp \,|\, \bX \in \mbX_0] \, }_{\infty} \,\geq\, \dfrac{\epsilon}{3}
\;\Big|\; \bX \in \mbX_0\right)
&\leq& 2 \, \exp\left( - \dfrac{2 \, M \, \epsilon^2}{9 \, \dep(\mbX_0)} + \log(1+p) \right),
\ee
as the definition of $\dep(\mbX_0)$ using \eqref{eq:delta_r} holds with probability $1$ 
when conditioning on the event $\bX \in \mbX_0$. 
Choosing 
\beno
\epsilon
\= \sqrt{\dfrac{27}{2}} 
\sqrt{\dfrac{\dep(\mbX_0) \; \log(\max\{M, 1+p\})}{M}}
\ee
establishes that 
\beno
\mbP\left(\norm{ \, \emp - \mbE[\emp \,|\, \bX \in \mbX_0] \, }_{\infty} \,\geq\, 
\sqrt{\dfrac{27}{2}} \, 
\sqrt{\dfrac{\dep(\mbX_0) \; \log(\max\{M, 1+p\})}{M}}
\;\Big|\; \bX \in \mbX_0\right)
&\leq& \dfrac{4}{\max\{M, 1+p\}^2}. 
\ee
Under the assumption that 
\beno
r(N)
&\leq& \sqrt{\dfrac{\dep(\mbX_0) \; \log(\max\{M, 1+p\})}{M}}
&<& \sqrt{\dfrac{3}{2}} \, \sqrt{\dfrac{\dep(\mbX_0) \, \log(\max\{M, 1+p\})}{M}}
\= \dfrac{\epsilon}{3},  
\ee
we have $\one(r(N) \,\geq\, \epsilon \,/\, 3) = 0$,
which allows us to  
revisit \eqref{eq:pbound1} to obtain the bound  
\be
\label{eq:last}
\mbP\left( \norm{\emp - \truth}_{\infty} \,\geq\, \sqrt{\dfrac{27}{2}} \, 
\sqrt{\dfrac{\dep(\mbX_0) \; \log(\max\{M, 1+p\})}{M}} \right)
&\leq& r(N) + \dfrac{4}{\max\{M, 1+p\}^2}. 
\ee
We lastly revisit \eqref{eq:last_step}
with \eqref{eq:last} 
to obtain the final bound 
\beno
\mbP\left( \norm{\emp - \truth}_{\infty} \,<\, 
\sqrt{\dfrac{27}{2}} \, 
\sqrt{\dfrac{\dep(\mbX_0) \, \log(\max\{M, 1+p\})}{M}}
\right)
&\geq& 1 - r(N) - \dfrac{4}{\max\{M, 1+p\}^2}. 
\ee

\end{proof}

\s\s

\begin{proof}[\textbf{\upshape Proof of Corollary \ref*{cor:1}:}] 
The assumptions of Theorem \ref{thm:main2} 
are met under the assumptions of Corollary \ref{cor:1},
and we may apply  Theorem \ref{thm:main2} to obtain the existence of a constant $N_0 \geq 1$
such that,
for all $N \geq N_0$,    
\beno
\mbP\left( \norm{\emp - \truth}_{\infty} \,<\, 
\sqrt{\dfrac{3}{2}} \,
\sqrt{\dfrac{\dep(\mbX_0) \, \log(\max\{M, \, 1+p\})}{M}} \right)
&\geq& 1  - r(N) - \dfrac{4}{\max\{M, \, 1+p\}^2}.
\ee
Noting that $M = N$, $p = N-1$, and $r(N) = 2 \,/\, N^2$, 
in this example, 
we obtain 
\beno
\mbP\left( \norm{\emp - \truth}_{\infty} \,<\, \sqrt{\dfrac{3}{2}} \, 
\sqrt{\dfrac{\dep(\mbX_0) \, \log(N)}{N}} \right)
&\geq& 1 - \dfrac{6}{N^2},
&&  N \geq N_0.  
\ee 
Under the assumption of both \eqref{eq:cor1_assumption_a} and \eqref{eq:cor1_assumption_b}, 
we use the bound on $\dep(\mbX_0)$ presented in \eqref{eq:cor1_dep_bound} to obtain 
\be
\label{eq:cor1_prob_bound}
\mbP\left( \norm{\emp - \truth}_{\infty} \,<\, 
(1 + M_{\max} + \alpha_{\max}) \, 
\sqrt{\dfrac{3}{2}} \,
\sqrt{\dfrac{\log(N)}{N}} \right)
&\geq& 1 - \dfrac{6}{N^2},
&& N \geq N_0. 
\ee
We establish the asymptotic convergence result utilizing the Borel-Cantelli lemma
(e.g., Theorem 4.1.3 of \citep{discrete_prob}). 
Define 
\beno
\epsilon_N
&\coloneqq& 
(1 + M_{\max} + \alpha_{\max}) \,
\sqrt{\dfrac{3}{2}} \,
\sqrt{\dfrac{\log(N)}{N}}, 
&& N \in \{N_0, N_0 + 1, \ldots\},
\ee
and note that $\epsilon_N \to 0$ as $N \to \infty$,
by the assumption that 
$M_{\max} + \alpha_{\max} = o\left(\hspace{-.1cm}\sqrt{N / \log(N)} \hspace{.05cm}\right)$. 
Leveraging \eqref{eq:cor1_prob_bound},
\beno
\dsum_{N=1}^{\infty} \, \mbP\left( \norm{\emp - \truth}_{\infty} \,\geq\, \epsilon_N \right)
&\leq& N_0 + \dsum_{N=N_0}^{\infty} \, \dfrac{6}{N^2}
&\leq& N_0 + \dsum_{N=1}^{\infty} \, \dfrac{6}{N^2}
&=& N_0 + \pi^2 
&<& \infty,
\ee
establishing 
through
Theorem 4.1.3 of \citep{discrete_prob}
that $\norm{\emp - \truth}_{\infty}$ converges almost surely to $0$ 
as $N \to \infty$. 
\end{proof}

\s

\begin{proof}[\textbf{\upshape Proof of Corollary \ref*{cor:bern}:}] 
By Definition \ref{def2},
\beno
\Delta_{N} \= 
\dfrac{1}{N}
\dsum_{i=1}^{N} \,
\dsum_{d=0}^{N-1} \,
\mbP\left(d_i = d\right) \,
\dsum_{j \in \{1, \ldots, N\} \setminus \{i\}} \,
\left(
\mbP\left(d_j = d \,|\, d_i = d\right) - \mbP\left(d_j = d\right)
\right),
\ee
where $d_i$ denotes the degree of node $i \in \mN$. 
By the law of total probability, 
\beno
\mbP\left(d_j = d \,|\, d_i = d\right)
\= \mbP\left(\{d_j = d\} \cap \{X_{i,j} = 1\} \,|\, d_i = d\right)
+ \mbP\left(\{d_j = d\} \cap \{X_{i,j} = 0\} \,|\, d_i = d\right),
\ee 
where each term can be expressed as
\beno
\mbP\left(\{d_j = d\} \cap \{X_{i,j} = 1\} \,|\, d_i = d\right)
\= \mbP\left( X_{i,j} = 1 \,|\, d_i = d \right) \, 
\mbP\left( d_j = d \,|\, X_{i,j} = 1 \right) \s\\ 
\mbP\left(\{d_j = d\} \cap \{X_{i,j} = 0\} \,|\, d_i = d\right)
\= \mbP\left( X_{i,j} = 0 \,|\, d_i = d \right) \,
\mbP\left( d_j = d \,|\, X_{i,j} = 0 \right), 
\ee
noting that event $\{d_j = d\}$ is conditionally independent of event $\{d_i = d\}$ 
when conditioning on event $\{X_{i,j} = 1\}$ or event $\{X_{i,j} = 0\}$ 
due to the independence of edges under a Bernoulli random graph.  
Next,
by the law of total probability,  
\beno
\mbP\left(d_j = d\right)
\= \mbP\left(d_j = d \,|\, X_{i,j} = 1 \right) \, \mbP\left(X_{i,j} = 1\right)
+ \mbP\left(d_j = d \,|\, X_{i,j} = 0\right) \, \mbP\left(X_{i,j} = 0 \right).  
\ee
Then,  
\beno
\mbP\left(d_j = d \,|\, d_i = d\right) - \mbP\left(d_j = d\right)
\= 
\mbP\left( d_j = d \,|\, X_{i,j} = 1 \right) \left(
\mbP\left( X_{i,j} = 1 \,|\, d_i = d \right) \,
- \mbP\left(X_{i,j} = 1\right) \right) \s\\
&& + \;\; \mbP\left( d_j = d \,|\, X_{i,j} = 0 \right) \left(
\mbP\left( X_{i,j} = 0 \,|\, d_i = d \right)  - 
\mbP\left(X_{i,j} = 0 \right)
\right). 
\ee
Define probability measures $\mbB_{\min}$ and $\mbB_{\max}$ for $\bX$ by 
\beno
\mbB_{\min}(X_{i,j} = 1) 
\;\coloneqq\; \min\limits_{\{a,b\} \subset \mN} \, \mbP(X_{a,b} = 1) \;>\; 0, 
&& \mbB_{\max}(X_{i,j} = 1)
\;\coloneqq\; \max\limits_{\{a,b\} \subset \mN} \, \mbP(X_{a,b} = 1) \;>\; 0,
&& \{i,j\} \subset \mN.
\ee
Noting that edge variables are exchangeable under both $\mbB_{\min}$ and $\mbB_{\max}$,
we then have the bounds  
\beno
\mbP\left( X_{i,j} = 1 \,|\, d_i = d \right)
&\leq& \mbB_{\max}\left(X_{i,j} = 1 \,|\, d_i = d \right)
\= \dfrac{d}{N-1} \s\\
\mbP\left( X_{i,j} = 0 \,|\, d_i = d \right)
&\leq& \mbB_{\min}\left(X_{i,j} = 0 \,|\, d_i = d\right)
\= 1 - \dfrac{d}{N-1}, 
\ee
which leads to the following upper bound on $\mbP\left(d_j = d \,|\, d_i = d\right) - \mbP\left(d_j = d\right)$: 
\beno
\mbP(d_j = d \,|\, X_{i,j} = 1) \, \left( \dfrac{d}{N-1} - \mbP(X_{i,j} = 1) \right) 
+ \mbP(d_j = d \,|\, X_{i,j} = 0) \, \left(\mbP(X_{i,j} = 1) - \dfrac{d}{N-1} \right) 
&\leq& \dfrac{d}{N-1} + \mbP(X_{i,j} = 1),
\ee
because $1 - d \, (N-1)^{-1} - \mbP(X_{i,j} = 0) = \mbP(X_{i,j} = 1) - d \,(N-1)^{-1}$.
Altogether, 
\beno
\Delta_N 
&\leq& \dfrac{1}{N} \dsum_{i=1}^{N} \, \dsum_{d = 0}^{N-1} \mbP(d_i = d) 
\dsum_{j \in \mN \setminus \{i\}} \left[ \dfrac{d}{N-1} + \mbP(X_{i,j} = 1) \right] 
\= \dfrac{1}{N} \dsum_{i=1}^{N} \left[ \mbE \, d_i + \dsum_{d = 0}^{N-1} \mbP(d_i = d) \, \mbE \, d_i \right] 
\;\;=\;\; \dfrac{2}{N} \dsum_{i=1}^{N} \mbE \, d_i. 
\ee
\end{proof}

\s

\begin{proof}[\textbf{\upshape Proof of Corollary \ref*{cor:2}:}] 
The assumptions of Theorem \ref{thm:main2}
are met under the assumptions of Corollary \ref{cor:2},
and we may apply  Theorem \ref{thm:main2} to obtain the existence of a constant $N_0 \geq 3$
such that,
for all $N \geq N_0$,
\beno
\mbP\left( \norm{\emp - \truth}_{\infty} \,<\,
\sqrt{\dfrac{3}{2}} \,
\sqrt{\dfrac{\dep(\mbX_0) \, \log(\max\{M, \, 1+p\})}{M}} \right)
&\geq& 1  - r(N) - \dfrac{4}{\max\{M, \, 1+p\}^2}.
\ee
By  assumption \eqref{deg_control},
$\mbP(\norm{\bX}_1 \,\geq\, N^{\beta}) \geq 1 - 2 \,/\, N^2$.
Since $M \geq N^{\beta}$ and $p = N-2$, 
we have,
for all $N \geq N_0 \geq 3$, 
that 
\beno
\mbP\left( \norm{\emp - \truth}_{\infty} \,<\,
\sqrt{\dfrac{3}{2}} \,
\sqrt{\dfrac{\dep(\mbX_0) \, \log(N)}{N^{\beta}}} \right)
&\geq& 1 - \dfrac{2}{N^2}- \dfrac{4}{(N-1)^2}
&\geq& 1 - \dfrac{11}{N^2},
\ee 
using the inequalities $\log(N-1) \leq \log(N)$ ($N \geq 1$) and  
$4 \,/\, (N-1)^2 \leq 9 \,/\, N^2$ (valid for $N \geq 3$). 
Under the assumption of both \eqref{eq:cor2_assumption_a} and \eqref{eq:cor2_assumption_b}, 
we use the bound on $\dep(\mbX_0)$ presented in \eqref{eq:cor2_dep_bound} to obtain 
\be
\label{eq:cor2_prob_bound}
\mbP\left( \norm{\emp - \truth}_{\infty} \,<\, 
(1 + M_{\max} + \alpha_{\max}) \, 
\sqrt{\dfrac{3}{2}} \,
\sqrt{\dfrac{\log(N)}{N^{\beta}}} \right)
&\geq& 1 - \dfrac{11}{N^2},
&& N \geq N_0. 
\ee
We establish the asymptotic convergence result utilizing the Borel-Cantelli lemma
(e.g., Theorem 4.1.3 of \citep{discrete_prob}). 
Define 
\beno
\epsilon_N
&\coloneqq& 
(1 + M_{\max} + \alpha_{\max}) \,
\sqrt{\dfrac{3}{2}} \,
\sqrt{\dfrac{\log(N)}{N^{\beta}}}, 
&& N \in \{N_0, N_0 + 1, \ldots\},
\ee
and note that $\epsilon_N \to 0$ as $N \to \infty$,
by the assumption that 
$M_{\max} + \alpha_{\max} = o\left(\hspace{-.1cm}\sqrt{N^{\beta} / \log(N)} \hspace{.05cm}\right)$. 
Leveraging \eqref{eq:cor2_prob_bound},
\beno
\dsum_{N=1}^{\infty} \, \mbP\left( \norm{\emp - \truth}_{\infty} \,\geq\, \epsilon_N \right)
&\leq& N_0 + \dsum_{N=N_0}^{\infty} \, \dfrac{11}{N^2}
&\leq& N_0 + \dsum_{N=1}^{\infty} \, \dfrac{11}{N^2}
&=& N_0 + \dfrac{11 \, \pi^2}{6} 
&<& \infty,
\ee
establishing 
through
Theorem 4.1.3 of \citep{discrete_prob}
that $\norm{\emp - \truth}_{\infty}$ converges almost surely to $0$ 
as $N \to \infty$. 
\end{proof}

\s\s

\begin{lemma}
\label{lem:AB_bound}
Let $A$, $B$, and $C$ be random variables. 
Then,
for all $t > 0$,  
\beno
\mbP(A+B+C \geq t) 
&\leq& \mbP\left(A \geq \dfrac{t}{3} \right) 
+ \mbP\left(B \geq \dfrac{t}{3} \right) 
+ \mbP\left(C \geq \dfrac{t}{3} \right).  
\ee
\end{lemma}

\begin{proof}[\textbf{\upshape Proof of Lemma \ref*{lem:AB_bound}:}]
Note that 
\beno
\left\{A \,<\, \dfrac{t}{3} \right\} \,\cap\, \left\{ B \,<\, \dfrac{t}{3} \right\}
\,\cap\, \left\{ C \,<\, \dfrac{t}{3} \right\}
&&\implies&& 
\left\{A + B +C < t \right\}.  
\ee
As a result, 
De'Morgan's law and a union bound shows that 
\beno
\mbP\left( A + B +C < t \right)
&\geq& \mbP\left( \left\{A \,<\, \dfrac{t}{3} \right\} \,\cap\, \left\{ B \,<\, \dfrac{t}{3} \right\} 
\,\cap\, \left\{ C \,<\, \dfrac{t}{3} \right\}
\right)
&\geq& 1 - \mbP\left( A \,\geq\, \dfrac{t}{3} \right) - \mbP\left( B \,\geq\, \dfrac{t}{3} \right)
- \mbP\left( C \,\geq\, \dfrac{t}{3} \right). 
\ee
Re-arranging terms in the expression proves the desired inequality  
shows that 
\beno
\mbP(A+B+C \geq t)
&\leq&
\mbP\left( A \,\geq\, \dfrac{t}{3} \right) + \mbP\left( B \,\geq\, \dfrac{t}{3} \right)
+ \mbP\left( C \,\geq\, \dfrac{t}{3} \right). 
\ee
\end{proof}

\bibliographystyle{myjmva}
\bibliography{library}

\end{document}